\documentclass[11pt,a4paper,reqno]{amsart}
\usepackage{amsmath,amssymb,amsfonts,epsfig,mathrsfs,cite,lineno,hyperref}
\usepackage[T1]{fontenc}
\usepackage{color}
\usepackage{array}
\usepackage{amsthm}
\usepackage{amstext}
\usepackage{graphicx}
\usepackage{setspace}

\usepackage[title]{appendix}

\makeatletter
\@namedef{subjclassname@2020}{%
  \textup{2020} Mathematics Subject Classification}
\makeatother

\usepackage[margin=2.5cm]{geometry}
\usepackage{color}
\usepackage{enumitem}
\usepackage{amscd,psfrag}
\usepackage{yhmath}
\usepackage[mathscr]{eucal}
\usepackage{comment}

\allowdisplaybreaks[4]

\usepackage{slashed}

\makeatletter
\pdfpageheight\paperheight
\pdfpagewidth\paperwidth

\usepackage{epstopdf}
\usepackage{chngcntr}
\counterwithin{figure}{section}
\usepackage{mathrsfs}

\usepackage{indentfirst}	

\usepackage[normalem]{ulem}
\theoremstyle{plain}
\numberwithin{equation}{section}

\newtheorem{definition}{Definition}[section]
\newtheorem{theorem}[definition]{Theorem}
\newtheorem*{theorem*}{Theorem}

\newtheorem{remark}[definition]{Remark}

\newtheorem*{remark*}{Remark}
\newtheorem*{sideremark*}{Side Remark}

\newtheorem*{claim*}{Claim}
\newtheorem*{q*}{Question}
\newtheorem{lemma}[definition]{Lemma}
\newtheorem{corollary}[definition]{Corollary}
\newtheorem*{corollary*}{Corollary}

\newtheorem{proposition}[definition]{Proposition}

\newcommand{\R}{\mathbb{R}}

\newcommand{\na}{\nabla}
\newcommand{\emb}{\hookrightarrow}

\newcommand{\id}{{\rm Id}}
\newcommand{\p}{\partial}

\newcommand{\e}{\varepsilon}

\newcommand{\dd}{{\rm d}}

\newcommand{\T}{{\mathcal{T}}}

\def\XXint#1#2#3{{\setbox0=\hbox{$#1{#2#3}{\int}$ }
\vcenter{\hbox{$#2#3$ }}\kern-.6\wd0}}

\def\XXint#1#2#3{{\setbox0=\hbox{$#1{#2#3}{\int}$ }
\vcenter{\hbox{$#2#3$ }}\kern-.6\wd0}}

\title{Spacetime decay of mild solutions and conditional quantitative transfer of regularity of the incompressible Navier--Stokes Equations from $\mathbb{R}^n$ to bounded domains}

\author{Siran Li}

\address{Siran Li: School of Mathematical Sciences $\&$ CMA-Shanghai, Shanghai Jiao Tong University, No.~6 Science Buildings,
800 Dongchuan Road, Minhang District, Shanghai, China (200240)}

\email{\texttt{siran.li@sjtu.edu.cn}}

\author{Xiangxiang Su}
\address{Xiangxiang Su: School of Mathematical Sciences, Shanghai Jiao Tong University, No.~6 Science Buildings,
800 Dongchuan Road, Minhang District, Shanghai, China (200240)}
\email{\texttt{sjtusxx@sjtu.edu.cn}}


\date{\today}

\begin{document}

\begin{abstract}
This paper is motivated by the "transfer of regularity" phenomenon for the incompressible Navier--Stokes Equations (NSE) in dimension $n \geq 3$; that is, the strong solutions of NSE on $\R^n$ can be nicely approximated by those on sufficiently large domains $\Omega \subset \R^n$ under the no-slip boundary condition. Based on the spacetime decay estimates of mild solutions of NSE established by Miyakawa~\cite{MR1815476}, Schonbek~\cite{MR0775190} and others, we obtain quantitative estimates on higher-order derivatives of velocity and pressure for the incompressible Navier--Stokes flow on large domains under certain additional smallness assumptions of the Stokes' system and/or the initial velocity, thus complementing the results obtained by Robinson \cite{MR4321417} and O\.z\'anski \cite{MR4316125}. 
\end{abstract}
\maketitle

\section{Introduction}\label{sec: intro}
The well-posedness of the incompressible Navier--Stokes equations (NSE) is a long-standing central problem in mathematical hydrodynamics and Partial Differential Equations (PDE):
\begin{subequations}  \label{2024.10.31.1}
\begin{align}
{\partial }_{t}\mathbf{v} - {\Delta \mathbf{v}} + \nu ( {\mathbf{v} \cdot \nabla } ) \mathbf{v} + \nabla p &= 0 &&\text { in } \Omega\times [0, T],\label{2024.12.03.1}\\
\operatorname{div}\mathbf{v} &= 0 &&\text { in } \Omega \times [0, T],\\
\mathbf{v}|_{t=0} &= \mathbf{v}_{0} && \text { in }  \Omega.
\end{align}
\end{subequations}
The vector field $\mathbf{v}$ and scalar field $p$ represent the velocity and pressure of the fluid, respectively, and $\nu >0$ is the viscosity coefficient. Throughout this note, $\Omega$ is a domain in $\R^n$; $n \geq 3$.

The theoretical literature on the well-posedness theory of NSE is abundant; we only mention here the recent survey \cite{MR3068540} by Tao. In this contribution, we focus on one problem on the analysis of NSE arising primarily from experimental investigations. As the basis for various physical and numerical experiments --- most notably, Schiller \emph{et al} \cite{phys} and Kerr \cite{MR3828692} on the trefoil configurations of vorticity --- 
for sufficiently localised initial velocity ${\bf v}_0$, the NSE solutions on $\Omega \Subset \R^{3}$ is expected to nicely approximate that on $\mathbb{R}^3$, provided that $\Omega$ is large enough.

This phenomenon was rigorously analysed by Robinson \cite{MR4321417} with $\Omega = \ell \mathbb{T}^3= [-\ell,\ell]^3$, and was referred to as the ``transfer of regularity'' from $\R^3$ to $\Omega$ therein. More precisely, \cite[Theorem~7.1]{MR4321417} establishes that if the strong solution to NSE on $\R^3$ exists up to time $T$, then for $\ell>0$ sufficiently large, the strong solution to NSE on $\Omega = \ell \mathbb{T}^3$ subject to the same initial datum also exists up to $T$, and their difference tends to zero in $\bigcap_{1 \leq r < \infty}L^r_tH^1_x$ as $\ell \nearrow \infty$.

The first quantitative version of Robinson's result was obtained by O\.za\'nski \cite{MR4316125}, in which Euclidean balls $\Omega = B_R \subset \R^3$ were considered in place of the tori $\ell \mathbb{T}^3$:
\begin{theorem}\label{thm: Ozanski}
    Let ${R}_{0} > 0$, $M \geq 1$, and $a \in \lbrack 0,3/2)$. Assume that $\mathbf{v}_{0} \in {C}_{0}^{5}\left( {B}_{{R}_{0}};\R^3\right)$ is divergence-free. Suppose that $(\mathbf{v},p)$ is a strong solution of \eqref{2024.10.31.1} in $\R^3 \times [0,T]$ such that
\begin{align*}
\|\mathbf{v} ( t) \|_{H^5} + \|\mathbf{v} ( t) \|_{W^{5,\infty }} \leq M
\end{align*}
for all $t \in [0,T]$. Then for every $R \geq {R}_{0} + 1$ such that
\begin{align*}
R \ge C(a,M,{\bf v}_0)\mathrm{e}^{C_0 M^4T/a},
\end{align*}
there exists a unique strong solution $(\mathbf{w},\pi)$ to the problem~\eqref{2024.10.31.1} posed on ${B}_{R}$ with $\mathbf{w}\big| _{\partial B_R} = 0$ and the same initial datum $\mathbf{w}|_{t=0} =\mathbf{v}_{0}$, where $C_0 > 1$ is a universal constant. Moreover,  
\begin{align*}
&\| \nabla ( \mathbf{v}- \mathbf{w} ) ( t) \|_{L^2 ( B_R) }\le C(a,M,\mathbf{v}_{0})\mathrm{e}^{C_0 M^4 t}R^{-a} ,\\
&\| \nabla (p-\pi)\|_{{L}^{q}\left( {0,t;{L}^{2}\left( {B}_{R}\right) }\right) }\le C(a,M,\mathbf{v}_{0}){t}^{\frac{1}{q}}\mathrm{e}^{C_0 M^4 t}R^{-a}
\end{align*}
for each $t \in [0,T]$ and $q \in ( 1,\infty)$.
\end{theorem}

This result was proved in \cite{MR4316125} by first truncating the NSE solution to ${B}_{R-1}$, then introducing (via Bogovskiĭ's Lemma~\ref{the Bogovskiĭ lemma}) a corrector to cope with the loss of divergence-free property of the truncated velocity, and finally constructing a second corrector in the annular region $B_R \setminus {B}_{R - 1}$ to recover the NSE in $B_R$. The most technical part of the proof, as commented at the beginning of \cite[p.98]{MR4316125}, is to obtain spatial decay estimates for the strong solution ${\bf v}$. To this end, O\.za\'nski utilised the Caffarelli--Kohn--Nirenberg inequalities (motivated by Kukavica and Torres \cite{MR1855669, kt1, kt2}), together with the Fourier method which enables the passage from the decay of $\operatorname{curl} \bf v$ to that of $\bf v$ \cite[Lemma~4]{MR4316125}. This inevitably leads to the issue of ``loss of derivatives'', namely that one starts with the boundedness assumption for the $H^5 \cap W^{5,\infty}$-norm of ${\bf v}$ but ends up with the decay estimates for the $\dot{W}^{1,2}$-norm of ${\bf v}$. Nevertheless, as strong solutions of  NSE automatically lie in $C_t (H^5 \cap W^{5,\infty})_x$, the aforementioned result in \cite{MR4316125} imposes no extra conditions on ${\bf v}$.

Combining  O\.za\'nski's arguments 
\cite{MR4316125} outlined above with the theory of spacetime decay estimates for \emph{mild solutions} of NSE developed by Schonbek \cite{MR0775190}, Takahashi \cite{MR1692815}, Miyakawa \cite{MR1815476} and others, we establish our main result,  Theorem~\ref{main thm}. A primary motivation of this result is to provide additional quantitative information on the transfer of regularity results in \cite{MR4321417, MR4316125}, but under additional smallness/decay assumptions on the linearised system or the initial datum (hence the use of ``\emph{conditional} transfer of regularity'' in the title). To formulate it, we shall need to consider the  \emph{Stokes system}, \emph{i.e.}, the PDE obtained by dropping the nonlinear convection term in NSE, subject to the same initial datum:
\begin{subequations}  \label{Stokes problem}
\begin{align}
{\partial }_{t}{\bf v}_{\rm St} - {\Delta{\bf v}_{\rm St}} + \nabla {p}_{\rm St} &= 0 &&\text { in } \R^n\times [0, T],\\
\operatorname{div}{\bf v}_{\rm St} &= 0 &&\text { in } \R^n \times [0, T],\\
{\bf v}_{\rm St}\big|_{t=0} &= \mathbf{v}_{0} && \text { at }  \R^n \times \{0\}.
\end{align}
\end{subequations}

Here and hereafter, $\mathbf{v}_{0} \in {C}_{0}^{k_0} ({B}_{R_0}; \R^{n})$ is identified with its extension-by-zero on $\R^n$ without further explanation. We fix $$\nu=1$$ in NSE~\eqref{2024.10.31.1} and assume that the dimension $n \geq 3$. A (Leray--Hopf) \emph{weak solution} ${\bf v}$ is in the regularity class $L^\infty\big(0,T; L^2(\Omega;\R^n)\big) \cap L^2\big(0,T; H^1(\Omega;\R^n)\big)$, and a \emph{strong solution} is in $L^\infty\big(0,T; H^1(\Omega;\R^n)\big) \cap L^2\big(0,T; H^2(\Omega;\R^n)\big)$. For $\Omega = \R^n$, the strong solutions are indeed smooth. See Leray's classical work~\cite{leray} and the modern exposition~\cite{op}.

\begin{theorem} \label{main thm}
Given $T,R_0>0$, $q \in (1,\infty)$, and $k_0 \in \mathbb{N}$ such that $q\cdot k_0>n$. Let $\mathbf{v}_{0} \in {C}_{0}^{k_0} ({B}_{R_0}; \R^{n})$ be a divergence-free vector field such that the solution ${\bf v}_{\rm St}$ of the Stokes system~\eqref{Stokes problem} satisfies
\begin{align}\label{2025.1.14.1}
&\left\|D^k{\bf v}_{\rm St}(\cdot,t) \right\|_{L^q( \R^n \smallsetminus B_{R_0})} \le C (R_0)^{-\alpha\left(n+1+k-\frac{n-1}{q}\right)}t^{-\beta\left(n+1+k- \frac{n}{q}\right) /2}\nonumber\\
&\qquad \text{for any $t \in [0,T]$, $k \in \{0,1,\ldots,k_0\}$, and } \alpha,\beta\geq 0 \text{ with } \alpha + \beta = 1.
\end{align}
There exists a constant $C_0>0$ such that the following holds whenever $C\leq C_0$ in \eqref{2025.1.14.1}:
\begin{itemize}
\item 
There exists a unique strong solution $(\mathbf{v},p)$ of NSE~\eqref{2024.10.31.1} in $\R^n \times [0,T]$. 
    \item 
There exists a unique strong solution $( \mathbf{w},\pi )$ to NSE~\eqref{2024.10.31.1} in $B_R \times [0,T]$ with the no-slip boundary condition $\mathbf{w}| _{\partial B_R} = 0$ and initial datum $\mathbf{w}|_{t=0} =\mathbf{v}_{0}$. Here $R \geq R_0+1$ is arbitrary. 
\item 
There is a constant  $C_1 >0$ such that 
\begin{eqnarray}\label{2024.10.31.2}
\begin{split}
\left\| D^k ( \mathbf{v}- \mathbf{w} ) (\cdot, t) \right\|_{L^q ( B_R) }\le C_1 R^{-\alpha\left(n-1+k-\frac{n}{q}\right)}t^{-\beta\left(n-1-\varepsilon+k- \frac{n}{q}\right) /2}\\
\text{ and }\quad \left\| D^k (p-\pi) (\cdot, t)\right\|_{L^q ( B_R) }\leq C_1 R^{-\alpha\left(n+k-\frac{n-1}{q}\right)}t^{-\beta\left(n+k- \frac{n}{q}\right)/2}
\end{split}
\end{eqnarray}
for any $k \in \{0,1,\ldots,k_0\}$ and $\alpha,\beta \geq 0$ with $\alpha + \beta = 1$.
\end{itemize}
In the above, $C_0$ and $C_1$ depend only on $T$, $R_0$, $q$, $n$, $k$, and the initial velocity ${\bf v}_0$. 
\end{theorem}

\begin{remark}
The assumption $q \cdot k_0 >n$ is only used towards the end of the proof (see \S\ref{sec: proof}), which ensures the validity of the Sobolev--Morrey embedding $W^{k_0, q}(\R^n)\emb L^\infty(\R^n)$, thus allowing us to apply a fixed point  argument to construct solutions to the nonlinear NSE. 
\end{remark}

Our Main Theorem~\ref{main thm} complements Theorem~\ref{thm: Ozanski} in several aspects below:
\begin{itemize}
    \item 
Instead of imposing decay assumptions on the NSE solution ${\bf v}(t,x)$, we only impose the decay assumption~\eqref{2025.1.14.1} on solutions to the Stokes problem (which is linear!) with the same initial datum ${\bf v}_0 \in C^{k_0}_c(\R^n;\R^n)$. %In particular, for a short time this can be ensured by the smallness condition $\|{\bf v}_0\|_{C^k} \ll 1$, which is \emph{a priori} in nature.
\item 
We obtain both spatial and temporal decay estimates of higher order for the difference between the NSE solutions on $\R^n$ and $B_R \subset \R^n$. 
\item 
We work with arbitrary dimensions $n \geq 3$. 
\end{itemize}

\begin{remark}
It should be emphasised that our Main Theorem~\ref{main thm} does \emph{not} serve as an improvement of Theorem~\ref{thm: Ozanski}, for the key point of the ``transfer of regularity'' results is to impose no additional decay or smallness assumptions on NSE/Stokes solutions or the initial data. Our  purpose is to showcase that the theory of spacetime decay for NSE mild solutions may shed new lights on the ``transfer of regularity'' or other related phenomena in mathematical hydrodynamics. %This explains our usage of ``\emph{conditional} transfer of regularity'' in the title.
\end{remark}

%Note that in the case $(n,q,k,k_0)=(3,2,1,5)$, Theorem~\ref{main thm}  implies $\left\| \na ( \mathbf{v}- \mathbf{w} ) ( t) \right\|_{L^2( B_R) }\lesssim R^{-2}$, which improves upon the decay rate $R^{-a}$ with $0\leq a < 3/2$ in Theorem~\ref{thm: Ozanski} (O\.z\'anski \cite{MR4316125}). The decay estimates in time and for higher-order spacetime derivatives are also established. 

In Theorem~\ref{thm: Ozanski}, the spatial decay rate of the $L^2$-norm of the velocity gradient is proved by using Caffarelli–Kohn–Nirenberg inequalities and Fourier analysis. Our proof of Theorem~\ref{main thm} adopts a different approach: we resort to the theory of spacetime decay for mild solutions of NSE, developed by Miyakawa \cite{MR1815476}, Schonbek \cite{MR0775190} and others.

%The proof of Theorem~\ref{main thm} in this note circumvents the use of Caffarelli--Kohn--Nirenberg inequalities and Fourier methods. Instead, we resort to the theory of spacetime decay for mild solutions of NSE, developed by Miyakawa  \cite{MR1815476}, Schonbek \cite{MR0775190} and others. Better decay results under suitable smallness assumptions. 

Let us comment on the condition~\eqref{2025.1.14.1}. Indeed, it is known that for any solenoidal, compactly supported vector field ${\bf v}_0 \in C^{k_0}_c(\R^n;\R^n)$, one has the pointwise estimate (\emph{cf}. Miyakawa \cite{MR1815476} or Schonbek \cite[Theorem~3.2]{MR3288012}): $$
\left| D^k \mathbf{v}_{\rm St} (x,t) \right| \leq C (1+|x|)^{-a} (1 + t)^{-b /2} \qquad \text{ for all } a, b\geq 0 \text{  with } a + b = n + 1 + k,$$ from which the $L^q$-estimate~\eqref{2025.1.14.1} follows directly. Moreover, by linearity of the Stokes system~\eqref{Stokes problem}, if ${\bf v}_{\rm St}(x,t)$ is a solution with the initial datum ${\bf v}_0$, then $\lambda {\bf v}_{\rm St}(x,t)$ is a solution with the initial datum $\lambda {\bf v}_0$ with any $\lambda>0$.

The observations above lead to the following ``transfer of regularity'' result, only assuming the smallness of the compactly supported initial velocity:
\begin{corollary}\label{cor}
Given $T,R_0>0$, $q \in (1,\infty)$, and $k_0 \in \mathbb{N}$ such that $q\cdot k_0>n$. Let $\mathbf{v}_{0} \in {C}_{0}^{k_0} ({B}_{R_0}; \R^{n})$ be a divergence-free vector field. There exists a constant $\e_0>0$ such that if $$\|{\bf v}_0\|_{C^{k_0}} \leq \e_0,$$ then the following holds:
\begin{itemize}
\item 
There exists a unique strong solution $(\mathbf{v},p)$ of NSE~\eqref{2024.10.31.1} in $\R^n \times [0,T]$. 
    \item 
There exists a unique strong solution $( \mathbf{w},\pi )$ to NSE~\eqref{2024.10.31.1} in $B_R \times [0,T]$ with the no-slip boundary condition $ \mathbf{w}| _{\partial B_R} = 0$ and initial datum $\mathbf{w}|_{t=0} =\mathbf{v}_{0}$. Here $R \geq R_0+1$ is arbitrary. 
\item 
There is a constant  $C_1>0$ such that 
\begin{eqnarray}
\begin{split}
\left\| D^k ( \mathbf{v}- \mathbf{w} ) (\cdot, t) \right\|_{L^q ( B_R) }\le C_1 R^{-\alpha\left(n-1+k-\frac{n}{q}\right)} t^{-\beta \left(n-1-\varepsilon+k-\frac{n}{q}\right) /2}\\
\text{ and }\quad \left\| D^k (p-\pi) (\cdot, t)\right\|_{L^q ( B_R) }\leq C_1 R^{-\alpha\left(n+k-\frac{n-1}{q}\right)}t^{-\beta\left(n+k- \frac{n}{q}\right)/2}
\end{split}
\end{eqnarray}
for any $k \in \{0,1,\ldots,k_0\}$ and $\alpha,\beta \geq 0$ with $\alpha + \beta = 1$.
\end{itemize}
Here $\e_0$ depends only on $T$, $R_0$, $q$, and $k_0$, while $C_1$ depends only on $T$, $R_0$, $q$, $n$, $k$, and $\e_0$. 
\end{corollary}

\begin{proof}[Proof of Corollary~\ref{cor}]

Under the smallness assumption for the initial velocity, existence of the unique strong solution up to a finite time $T>0$ is well-known (with $\e_0$ depending on $T$), and the decay estimate~\eqref{2025.1.14.1} on the linear Stokes system, \emph{i.e.}, the assumption in Theorem~\ref{main thm}, follows from \cite{MR0775190, MR1692815, MR1815476}. Hence we may conclude  from Theorem~\ref{main thm}.   \end{proof}

\begin{remark}
Analogous results to Theorem~\ref{main thm} and Corollary~\ref{cor} for NSE on a bounded smooth domain $\Omega \Subset \R^n$ hold too: we may simply take $R_0$ so large that $\Omega \Subset B_{R_0}$.  On the other hand, we expect analogous results to Theorem~\ref{main thm} and Corollary~\ref{cor} also hold in dimension $n=2$: one may need to utilise instead the 2D spacetime decay estimates for mild solutions of the NSE and the Stokes system. We leave it for future investigations.
\end{remark}

To put it in perspectives, let us remark that various investigations have been carried out in the literature under the theme of ``transfer of regularity''. Constantin (1986) proved in the seminal work \cite{MR0836008} that, on $\mathbb{R}^3$ or $\mathbb{T}^3$, if the smooth solution to the Euler equations (\emph{i.e.}, the PDE obtained by formally taking the viscosity coefficient as zero in NSE) exists up to time $T$, then NSE with the same initial datum and sufficiently small viscosity $\nu$ also admits a unique smooth solution by time $T$. For $\Omega = \ell \mathbb{T}^3$, there exists a critical viscosity $\nu_s \sim \ell^{-3}$ such that when $\nu < \nu_s$, the higher-order norms of the NSE solution are controlled by those of the Euler solution. Heywood \cite{h} showed that a Leray--Hopf weak solution of NSE on $\R^3$ can be realised as a limit of weak solutions on $B_R$. See also the subsequent developments by Kelliher \cite{k}, Chernyshenko--Constantin--Robinson--Titi \cite{ccrt}, and Enciso,  Garc\'ia-Ferrero, and D. Peralta-Salas \cite{MR3934592, MR4315668}.

The remaining parts of this note are organised as follows. Section~\ref{sec: PDE} introduces the PDE systems (the same as in \cite{MR4316125}) that shall be used throughout. Section~\ref{sec: prelim} sets up the framework of mild solutions and derives preliminary estimates for $({\bf v}, p)$. Next, the estimates for the truncated velocity and the first corrector are given in Section~\ref{sec: est}. An iterative estimate is proved in Section~\ref{sec: iteration} to bound the second corrector function, which will be used in the fixed point argument leading to the existence of solutions in later parts. Finally, we conclude the proof of Theorem~\ref{main thm} in Section~\ref{sec: proof}. The Bogovskiĭ Lemma and relevant estimates for $e^{-tA}\mathbb{P}$ are presented in the Appendix.

Before further developments, we comment on the notation of norms throughout this note. For vector-valued function $\mathbf{v} = (v_1, v_2, \dots, v_n)$,  $1 \le q \le \infty$, and $|I| \in \mathbb{N}$, we write $$\left\|D^{I} \mathbf{v}\right\|_{L^q} = \sup_{j,I} \left\|D^{I} v_j\right\|_{L^q} =\mathop{\sup }\limits_{\substack{j},I} \left\|\partial_{x_1}^{i_1} \partial_{x_2}^{i_2} \cdots \partial_{x_n}^{i_n} v_j\right\|_{L^q}.$$ The supremum is taken over all multi-indices $I = (i_1, i_2, \ldots, i_n) \in \mathbb{N}^n$ with $\Sigma_{j=1}^n i_j=|I|$ and all $j \in \{1,2,\ldots,n\}$. Also, for a function $f = f(x,t)$ and $1<p,q<\infty$, we write $\|f\|_{L^p(0,T; L^q(\Omega))}:= \left\{\int_0^T\left(\int_\Omega |f(x,t)|^q\,\dd x\right)^{\frac{p}{q}}\,\dd t\right\}^{\frac{1}{p}}$ and $\|f(t)\|_{L^q(\Omega)} := \left\{\int_\Omega |f(x,t)|^q\,\dd x\right\}^{\frac{1}{q}}$. Moreover,
\begin{remark}\label{remark on const}
Throughout this note, unless otherwise specified, all the constants $C, C_1, C_2,\ldots$ depend only on $q$, $k$, $R_0$, and the dimension $n$, which may vary from line to line.
\end{remark}

\section{Derivation of the PDE}\label{sec: PDE}
Recall the Navier--Stokes Equations (NSE)~\eqref{2024.10.31.1}, in which we take $\nu=1$ once and for all:
\begin{equation*}
    {\partial }_{t}\mathbf{v} - {\Delta \mathbf{v}} + ( {\mathbf{v} \cdot \nabla } ) \mathbf{v} + \nabla p = 0\quad\text{and}\quad {\rm div}\,{\bf v}=0\qquad\text{in } \R^n \times [0,T]
\end{equation*}
with the initial datum $\mathbf{v}|_{t=0} = \mathbf{v}_{0}$ and dimension $n \geq 3$. In this note, we take a smooth bounded domain $\Omega \subset \R^{n}$ and assume that $\Omega \subset B_{R_0}$ for some $R_0 >0$. In what follows $R\geq R_0+1$.

Consider a spatial cutoff function: $\phi(x)=\tilde{\phi}(|x|)$ for some $\tilde{\phi} \in C^\infty(\R_+)$ such that $\phi=1$ within ${B}_{R - 1}$ and $\phi = 0$ outside $B_R$. We have the following PDE for the cutoff velocity $\phi \mathbf{v}$:
\begin{subequations}\label{eq for phi v}
\begin{align}
{\partial }_{t} ( \phi \mathbf{v} ) - \Delta ( \phi \mathbf{v} ) + \left( (\phi \mathbf{v}) \cdot \nabla \right) ( \phi \mathbf{v}) + \nabla ( \phi p )&= {F}_{1} + p \nabla \phi \;&&\text{ in }B_R \times [0,T], \\
\operatorname{div}( \phi \mathbf{v}) &= \nabla \phi \cdot \mathbf{v}  \;&&\text{ in }B_R \times (0,\infty),\\
\phi \mathbf{v} &= 0 \;  &&\text { on } \partial B_R \times [0,T],\\
( \phi \mathbf{v}) | _{t=0} &= \mathbf{v}_{0} \;&&\text{ at }B_R \times \{0\},
\end{align}
\end{subequations}
where
\begin{align} \label{2024.11.18.5}
{F}_{1} := - \phi ( 1 - \phi ) ( {\mathbf{v} \cdot \nabla }) \mathbf{v} + ( \phi \mathbf{v} \cdot \nabla \phi ) \mathbf{v} - \mathbf{v}\Delta \phi  - 2\nabla \mathbf{v} \cdot \nabla \phi .
\end{align}

The divergence-free property is not preserved under spatial cutoff. Thus, let us introduce a correction function $\mathbf{v}_{c}$ such that
\begin{align} \label{2024.11.11}
\mathbf{r} := \phi \mathbf{v} + \mathbf{v}_{c}
\end{align}
is divergence free and has zero trace on $\partial B_R$. 
Thus, $\mathbf{v}_{c}$ satisfy the PDE:
\begin{subequations} \label{2025.1.7.1}
\begin{align}
 {\partial }_{t} \mathbf{v}_{c}- \Delta \mathbf{v}_{c} + ( \mathbf{v}_{c} \cdot \nabla ) \mathbf{v}_{c} + ( \mathbf{v}_{c} \cdot \nabla ) ( \phi \mathbf{v}) + ( \phi \mathbf{v} \cdot \nabla ) \mathbf{v}_{c} &:= {F}_{2} \;&&\text{ in }B_R \times [0,T],\\
\operatorname{div} \mathbf{v}_{c}&=-\operatorname{div}( \phi \mathbf{v}) \;&&\text{ in }B_R \times [0,T].
\end{align}
\end{subequations}

To construct $\mathbf{w}$ and $\pi$ that satisfy NSE~\eqref{2024.10.31.1} as stated in Theorem~\ref{main thm}, we introduce further corrector functions $\tilde{\mathbf{v}}$ and $\bar{p}$ so that 
\begin{align} \label{2024.11.19.1}
\mathbf{w} = \mathbf{r} + \tilde{\mathbf{v}} \qquad \text{ and } \qquad \pi = \phi p + \bar{p}.
\end{align}
The PDE system for $\left(\tilde{\mathbf{v}}, \bar{p}\right)$ is as follows:
\begin{subequations} \label{2024.11.2.1}
\begin{align}
{\partial }_{t}\tilde{\mathbf{v}} - \Delta \tilde{\mathbf{v}} + ( \tilde{\mathbf{v}} \cdot \nabla ) \tilde{\mathbf{v}} + \nabla \bar{p } &= -F -p \nabla \phi- ( \tilde{\mathbf{v}} \cdot \nabla ) \mathbf{r} - ( \mathbf{r} \cdot \nabla ) \tilde{\mathbf{v}}\;&&\text{ in }B_R \times [0,T],\label{2025.01.03.1}\\
\operatorname{div}\tilde{\mathbf{v}} &= 0\;&&\text{ in }B_R \times [0,T], \\
\tilde{\mathbf{v}} &= 0\;  &&\text { on } \partial B_R \times [0,T],\\
\tilde{\mathbf{v}} | _{t=0} &= \mathbf{v}_{0} \;&&\text{ at }B_R \times \{0\},
\end{align}
\end{subequations}
where $F := {F}_{1} + {F}_{2}$. We emphasise that $F$ is supported in the annulus $B_R\smallsetminus B_{R-1}$.

\section{Preliminaries}\label{sec: prelim}

The analytical framework for this paper is based on the theory of \emph{mild solutions} of NSE~\eqref{2024.10.31.1}, \emph{i.e.}, the representation of (in our case, strong) solutions of NSE via the Stokes semigroup. One first considers the linearised equation, namely that the Stokes system:
\begin{align} \label{2024.11.18.1}
{\partial }_{t}\mathbf{v} - {\Delta \mathbf{v}} + \nabla p =f \quad \text{and} \quad \operatorname{div} \mathbf{v}=0\qquad \text{in } \Omega \times [0,T].
\end{align}

\subsection{Mild solution}  \label{sec:Mild solution}
Denote by $\mathbb{P}: L^2(\R^n; \R^n) \to L^2(\R^n; \R^n)$ the Leray projection operator, \emph{i.e.}, the $L^2$-orthonormal projection of vector fields on to the divergence-free part. %Notice that this projection commutes with derivatives. As a result, $\mathbb{P} f$ satisfies $\operatorname{div} \mathbb{P} f=0$. 
The \emph{Stokes operator} is $$A:=-\mathbb{P} \Delta.$$ Then we recast original Stokes equations \eqref{2024.11.18.1} into the ordinary differential equations:
$$
\frac{d}{d t} \mathbf{v}+A \mathbf{v}=\mathbb{P} f.
$$
The solution to the problem \eqref{2024.11.18.1} is given by the formula
\begin{align} \label{2024.11.18.2}
\mathbf{v}(t)=e^{-t A} \mathbf{v}_{0}+\int_0^t e^{-(t-s) A} \mathbb{P} f(s) \,\mathrm{d} s,
\end{align}
where $e^{-t A}$ is the semigroup operator generated by $A$, \emph{a.k.a.} the Stokes semigroup.

As in Miyakawa \cite{MR1815476}, by considering $\operatorname{div} ( \mathbf{v} \otimes \mathbf{v})$ as the source term $f$, we obtain the integral form of solutions of NSE \eqref{2024.10.31.1}: 
\begin{align} \label{2024.11.1.1}
\mathbf{v}( t) = e^{-tA}\mathbf{v}_{0} - \int _{0}^{t} \nabla \cdot e^{-(t - s)A} \mathbb{P}( \mathbf{v} \otimes \mathbf{v}) (s) \,\mathrm{d} s. 
\end{align}
In this case, $\mathbf{v}$ is said to be a mild solution of NSE \eqref{2024.10.31.1}. %where $B_R$ is generalized to any bounded domain $\Omega \in \mathbb{R}^{n}$.\\
%for the kernel function $K( x,t)$ of $\nabla e^{-tA}\mathbb{P}=(\partial_l e^{-tA}\mathbb{P}_{jk})$, it holds that
%\begin{eqnarray} 
%\begin{split}
%\mathbf{v}( t) &= e^{-tA}\mathbf{v}_{0} - \int _{0}^{t} \nabla \cdot e^{-(t - s)A} \mathbb{P}( \mathbf{v} \otimes \mathbf{v}) (s) \,\mathrm{d} s.
%&=e^{-tA}\mathbf{v}_{0} -\int _{0}^{t}\!\int K( x-y,t-s)( \mathbf{v} \otimes \mathbf{v}) (y,s) \mathrm{d} y\mathrm{d} s. 
%\end{split}
%\end{eqnarray} 
%$$
%\mathbf{v} \in {L}^{\infty }\left( ( 0,T) ;V \right) \cap {L}^{2}\left( (0,T) ;{H}^{2}( \Omega ) \right),
%$$
%where $V$ denotes the closure of the set of divergence-free functions $\boldsymbol{\phi} \in {C}_{0}^{\infty } ( \Omega  )$ with respect to the ${H}^{1}$ norm.\\
%ii) For any divergence-free test function $\boldsymbol{\varphi} \in {C}_{0}^{\infty } ( [0,\infty) \times \Omega )$ and almost every $s \in (0,T)$ , it holds that
%\begin{align*}
%&\int _{0}^{t}\! \int _{\Omega }\left( -\mathbf{v} \cdot {\partial }_{t}\boldsymbol{\varphi} + \nabla \mathbf{v} : \nabla \boldsymbol{\varphi} + \left( ( \mathbf{v} \cdot \nabla ) \mathbf{v} + ( \mathbf{v} \cdot \nabla ) \mathbf{r} + ( \mathbf{r} \cdot \nabla ) \mathbf{v} - F\right) \cdot \boldsymbol{\varphi} \right) \,\mathrm{d} x \mathrm{d} s\\
% =& \int _{\Omega } \mathbf{v}_{0} \cdot \boldsymbol{\varphi} ( 0) - \int_{\Omega }\mathbf{v}( t) \cdot \boldsymbol{\varphi} (t)  \,\mathrm{d} x.
%\end{align*}

\subsection{Decay rates of $\mathbf{v}$} 

The decay properties for the solutions of NSE in the whole space have been extensively studied; see Brandolese \cite{ MR2076682}, Kukavica \cite{MR1855669}, Schonbek \cite{MR0775190}, Takahashi \cite{MR1692815} and the many references cited therein. Our subsequent developments rely heavily on the spacetime decay estimates for the strong solutions of NSE together with their derivatives. To this end, we refer to the survey paper \cite{MR3288012} by Schonbek, which provides a concise summary of results \emph{\`{a} la} Miyakawa \cite{MR1815476} and establishes decay estimates for higher-order derivatives. 
\begin{proposition} \label{thm 2.1}
(\cite[Theorem 4.1]{MR3288012}) Let $\gamma = n + 1 + k$ with $k \in \{0,1,2,\ldots\}$, and let $\mathbf{v}_{0} \in {C}^{k}_0\left(\R^n;\R^n\right)$ be a soleinoidal vector field such that
\begin{align} \label{2024.11.6.1}
\left| {e}^{-tA} D^k \mathbf{v}_{0} (x) \right| \leq {C}_{0}\left( 1 + | x| \right)^{-\alpha } ( 1 + t) ^{-\beta /2}\qquad \text{for each } \alpha, \beta \geq 0 \text{ with } \alpha + \beta = \gamma.
\end{align}
If ${C}_{0}$ is small enough, then there exists a mild solution $\mathbf{v}$ of \eqref{2024.10.31.1} in $\mathbb{R}^n$ with initial datum $\mathbf{v}_{0}$ such that
\begin{align} \label{2024.11.6.2}
\left| D^k \mathbf{v} (x,t) \right| \leq C \left( 1 +| x| \right) ^{-\alpha } ( 1 + t)^{-\beta /2}\qquad \text{for each } \alpha, \beta \geq 0 \text{ with }\alpha + \beta = \gamma. 
\end{align}
The solution $\mathbf{v}$ satisfies the initial condition $\mathbf{v}\big|_{t = 0} = \mathbf{v}_{0}$ in the sense that
$$
\mathop{\lim }\limits_{t \searrow 0} D^{I} \mathbf{v} (x,t) = {D}^{I} \mathbf{v}_{0} ( x) \qquad \text{for  a.e. }x \in {\mathbb{R}}^{n}, 
$$
where $I \in \mathbb{N}^n$ is an arbitrary multi-index with $|I| \leq k$.
\end{proposition}

%For technical reasons, we employ a Bogovskiĭ-type correction. However, since the Bogovskiĭ Lemma does not provide pointwise estimates, it is necessary to obtain the $L^q$ norm. 

As a consequence of pointwise decay estimates, 
\begin{align} \label{2024.11.29.5}
\left\|D^k \mathbf{v}(\cdot ,t)\right\|_{L^q( B_R \smallsetminus {B}_{R - 1})}^q &\le C\int_{R-1}^R ( 1 +r ) ^{-q(n+1+k) } r^{n-1} \,\mathrm{d} r\nonumber\\
&\le C R^{-(qn+q+qk-n+1)}.
\end{align}
This implies that
\begin{align}\label{2024.11.29.1}
\left\|D^k \mathbf{v}(\cdot ,t)\right\|_{L^q( B_R \smallsetminus {B}_{R - 1})} \le C R^{-[n+1+k-(n-1)/q]}.
\end{align}
Similarly, we have that 
\begin{align} \label{2024.12.30.6}
\left\| D^k \mathbf{v}(\cdot ,t)\right\|_{L^q( B_R \smallsetminus {B}_{R - 1})}^q  &\le C t^{-[q(n+1+k)-n]/2} \int_{B_R \smallsetminus {B}_{R - 1}} |x|^{-n} \,\mathrm{d} x, 
\end{align}
and consequently
\begin{align} \label{2024.12.31.7}
\left\|D^k \mathbf{v}(\cdot ,t)\right\|_{L^q( B_R \smallsetminus {B}_{R - 1})} \leq C t^{-(n+1+k- \frac{n}{q}) /2} \text{ for } 1 \leq q \leq \infty.
\end{align}
%provided that $\varepsilon$ is positive.

\subsection{Estimates for pressure $p$} \label{section 1.1}
%The previous subsection estimates the spatial derivatives of $\mathbf{v}$. Since estimates for $\partial_t \mathbf{v}$ will also be required later, it is necessary to establish corresponding bounds for the pressure $p$.
In this subsection, we derive estimates for the pressure $p$ and its derivatives away from the origin, which is crucial to our arguments.

An application of the divergence operator to both sides of Equation \eqref{2024.12.03.1} yields
\begin{align*}
-\Delta p(x,t)=\operatorname{div}\operatorname{div}(\mathbf{v}\otimes \mathbf{v})(x,t).
\end{align*}
%Let $\zeta$ be a cut-off function taking the form: for any $R>0$,
%$$
%\begin{cases}
%\zeta(x)=0\quad \text { for } & |x|<R/4, \\ 
%\zeta(x)=1\quad \text { for } & |x| \geq R/2.
%\end{cases}
%$$
%Applying $\zeta$ to the two equations above, we deduce that
%\begin{align*}
%-\Delta (\zeta p)(x,t)=\left((\Delta \zeta)p +2\nabla \zeta \cdot \nabla p+\zeta \operatorname{div}\operatorname{div}(\mathbf{v}\otimes \mathbf{v})\right)(x,t),
%\end{align*}
%and
%\begin{align*}
%-\Delta (\zeta\nabla p)(x,t)=(\Delta \zeta) \nabla p +2\nabla \zeta \cdot \nabla^2 p+\zeta \nabla\operatorname{div}\operatorname{div}(\mathbf{v}\otimes \mathbf{v})(x,t).
%\end{align*}
Hence,
\begin{align}\label{Dkp}
&D^k p(x,t)=\int_{\mathbb{R}^{n}} C(n) \frac{1}{|x-y|^{n-2}} D^k \operatorname{div}\operatorname{div}(\mathbf{v}\otimes \mathbf{v}) (y,t) \,\mathrm{d}y.
\end{align} To investigate the pointwise estimate of $p$, we make use of the representation formula~\eqref{Dkp} and Proposition~\ref{thm 2.1}. Indeed, it follows from the pointwise decay estimate~\eqref{2024.11.6.2} of $\mathbf{v}$ that
$$
\left|D^k( \mathbf{v} \otimes \mathbf{v}) (y,t) \right| \leq \mathop{\sum }\limits_{j = 0}^{k} \left|D^j \mathbf{v}\right| \left|D^{k - j}\mathbf{v}\right| \leq C ( 1 + | y| )^{-( \alpha_{j} + \alpha_{k - j})} ( 1 + t) ^{-( \beta_{j}/2 + \beta_{k - j}/2)}
$$
$$
\text{with }\, \alpha_{j} + \alpha_{k - j} + \beta_{j} + \beta_{k - j} = 2(n + 1) + k.
$$
Thus, by integration by parts,
\begin{align*}
|p(x,t)| &\le C(n) \left(\int_{|x - y| \leq |x| /2} \frac{1}{|x-y|^{n-2}}\left|\operatorname{div}\operatorname{div}(\mathbf{v}\otimes \mathbf{v})\right| (y,t)\,\mathrm{d}y+\int_{|x - y| > |x| /2} \frac{\left|\mathbf{v}\otimes \mathbf{v} \right|(y,t)}{|x-y|^{n}} \,\mathrm{d}y\right)\\
&\le C(n)\left(\int_{|x - y| \leq |x| /2} | x - y|^{2 - n} ( 1 + |y| )^{-2n-4} \,\mathrm{d} y+\int_{|x - y| > |x| /2} | x - y|^{- n} ( 1 + |y| )^{-2n-2} \,\mathrm{d} y\right).
\end{align*}
In the region $|x - y| \leq |x| /2$, the inequality $|x| \leq |x - y| + |y| \leq |x|/2+ |y|$ holds true, which implies $ |x| \leq 2 |y|$. On the other hand, in the region $|x - y| > |x| /2$, we have $|x - y|^{-n} \le (|x|/2)^{-n}$, and $\int_{|x - y| > |x|/2} ( 1 + | y| )^{- 2n-2}\,\mathrm{d} y $ is finite. As a consequence,
\begin{align} \label{2024.12.03.2}
|p(x,t)| &\le  C(n) \left( |x|^{2} ( 1 + |x|)^{-2n-4}+ |x|^{-n} \right) \nonumber\\
&\le C(n)  ( 1 + |x|)^{-n}.
\end{align}

Similar arguments (but taking into account here the temporal decay of ${\bf v}$) lead to
\begin{align}
|p(x,t)| &\le C(n) \left|\int_{\mathbb{R}^{n}} \frac{1}{|x-y|^{n-2}}( 1 + |y| )^{-(n+2)}( 1 + t) ^{-\frac{n + 2}{2}} \mathrm{d}y\right|\notag\\
&\le C(n) \left(\int_{|x - y| \leq |x| /2}+\int_{|x - y| > |x| /2}  \right)\frac{1}{|x-y|^{n-2}}( 1 + |y| )^{-(n+2)}( 1 + t) ^{-\frac{n + 2}{2}} \mathrm{d}y\notag\\
&\le C(n)( 1 + t) ^{-\frac{n + 2}{2}}\left( |x|^{2} ( 1 + |x|)^{-(n+2)}+ |x|^{-(n-2)} \right)\notag\\
&\le C(n)( 1 + t) ^{-\frac{n + 2}{2}}.  \label{2024.12.08.1}
\end{align}
Putting \eqref{2024.12.03.2} and \eqref{2024.12.08.1} together, we bound
\begin{align} \label{2024.12.08.2}
|p(x,t)| \le C (1 +|x|) ^{-\alpha } ( 1 + t)^{-\beta /2}
\end{align}
for all $\alpha, \beta \geq 0 $ with $\alpha + \beta = n$.

Arguments analogous to those leading to the estimate~\eqref{2024.12.08.2} allow us to achieve
\begin{align} \label{2024.12.08.3}
\left|D^k p(x,t)\right| \le C (1 +|x|) ^{-\alpha } ( 1 + t)^{-\beta /2}
\end{align}
for all $\alpha, \beta \geq 0 $ with $\alpha + \beta = n+k$. Also, arguments parallel to those for \eqref{2024.11.29.5} and \eqref{2024.12.30.6} give us
\begin{equation}
    \begin{cases}\label{2024.12.08.4}
\|p(\cdot ,t)\|_{L^q( B_R \smallsetminus {B}_{R - 1})} \le C R^{-[n-(n-1)/q]},\\
\left\|D^k p(\cdot ,t)\right\|_{L^q( B_R \smallsetminus {B}_{R - 1})}  \le C R^{-[n+k-(n-1)/q]},
    \end{cases}    
\end{equation}
as well as
\begin{equation}\label{2025.1.7.8}
    \left\|D^k p(\cdot ,t)\right\|_{L^q( B_R \smallsetminus {B}_{R - 1})} \leq C t^{-(n+k- \frac{n}{q}) /2}\qquad\text{for any $1 \leq q \leq \infty$.}
\end{equation}

%Moreover, for $q \in (1,\infty)$, we have the following inequalities:
%\begin{align} \label{2024.11.18.3}
%\|p \|_{L^q}\le C(q) \| |\mathbf{v}| ^{2}\|_{L^q},\quad
%\| \nabla p \|_{L^q} \le C(q) \left\| | \mathbf{v}| | \nabla \mathbf{v}|\right \|_{L^q},
%\end{align}
%as proved in \cite[Lemma 4.1 and Lemma 4.2]{MR1855669}. Our result corresponds to the special case $\alpha = 0$ in their paper.

\section {Truncated velocity estimates and the Bogovskiĭ-type correction}\label{sec: est}

We now establish the decay of the mild solution within the ball and resolve the non-divergence-free issue of the truncated velocity $\phi \mathbf{v}$.

\subsection {Decay estimates of mild solution in a ball}
If the initial value $\mathbf{v}_0$ satisfies the condition~\eqref{2024.11.6.1}, then by  Proposition~\ref{thm 2.1} there exists a solution $\mathbf{v}$ exhibiting the pointwise decay property as in \eqref{2024.11.6.2}. As $\phi$ is  compactly supported, we easily deduce that
\begin{equation}\label{2025.1.13.1}
\left\|D^k (\phi \mathbf{v})(\cdot ,t)\right\|_{L^q( B_R \smallsetminus {B}_{R - 1})} \le \begin{cases}
C R^{-[n+1+k-(n-1)/q]},\\
C t^{-(n+1+k- \frac{n}{q}) /2},
\end{cases}
\end{equation}  
for any $1\leq q \leq \infty$.

Also, recall from the definition of $F_1$ in \eqref{2024.11.18.5} that $F_1 \equiv 0$ on $B_{R-1}$. Thus it is supported on the annulus ${B}_R \smallsetminus {B}_{R-1}$. The pointwise decay estimate~\eqref{2024.11.6.2} for $\mathbf{v}$ yields that
\begin{align} \label{2024.12.13.1}
|{F}_{1} (x,t)| \le& C\Big(\left| ( {\mathbf{v} \cdot \nabla }) \mathbf{v}  \right|+ | \mathbf{v}|^2 +|\mathbf{v}|+ |\nabla \mathbf{v} |\Big)\leq C \left( 1 +| x| \right) ^{-\alpha } ( 1 + t)^{-\beta /2}
\end{align}
for all $\alpha, \beta \geq 0 $ with $\alpha + \beta = n+1$. Consequently, arguments similar to those in \eqref{2024.11.29.5} lead to 
\begin{align*}
\|{F}_{1} (\cdot ,t)\|_{L^q( B_R)}^q = \|{F}_{1} (\cdot ,t)\|_{L^q( B_R \smallsetminus {B}_{R - 1})}^q \le C\left\| ( 1 +| x| ) ^{-(n+1)} \right\|_{L^q( B_R \smallsetminus {B}_{R - 1})}^q,
\end{align*}
which shows that
\begin{align} \label{2024.11.18.6}
\|{F}_{1} (\cdot ,t)\|_{L^q( B_R)} \le C R^{-[n+1-(n-1)/q]}.
\end{align}
Similarly, one deduces that
\begin{equation}
    \left|D^k {F}_{1} (x,t)\right| \le C \left( 1 +| x| \right) ^{-\alpha } ( 1 + t)^{-\beta /2}\qquad  \text{ with } \alpha + \beta = n+1+k
\end{equation}
and that
\begin{equation}\label{2024.12.23.2}
\left\|D^k{F}_{1} (\cdot ,t)\right\|_{L^q( B_R)} \le     \begin{cases}
C R^{-[n+1+k-(n-1)/q]},\\
 C t^{-(n+1+k- \frac{n}{q}) /2}.
    \end{cases}
\end{equation}
Here $q \in [1,\infty]$ is arbitrary. 

\subsection{Bogovskiĭ-type correction}
Recall the definition of the corrector function $\mathbf{v}_{c}$ ---  see Equation~\eqref{2025.1.7.1}, reproduced here: 
\begin{align*}
 {\partial }_{t} \mathbf{v}_{c}- \Delta \mathbf{v}_{c} + ( \mathbf{v}_{c} \cdot \nabla ) \mathbf{v}_{c} + ( \mathbf{v}_{c} \cdot \nabla ) ( \phi \mathbf{v}) + ( \phi \mathbf{v} \cdot \nabla ) \mathbf{v}_{c} &:= {F}_{2} \;&&\text{ in } B_R \times [0,T],\\
\operatorname{div} \mathbf{v}_{c}&=-\operatorname{div}( \phi \mathbf{v}) \;&&\text{ in }B_R \times [0,T].
\end{align*}
Recall also that the vector field
$$
\mathbf{r} := \phi \mathbf{v} + \mathbf{v}_{c}
$$ 
is divergence-free. Since $\operatorname{div}( \phi \mathbf{v})=0$ on ${B}_{R - 1}$, it follows that $\mathbf{v}_{c}$ vanishes on ${B}_{R - 1}$. Moreover, $F_2$ also vanishes within $B_{R-1}$. Thus we shall only need to consider the PDE for ${\bf v}_c$ on $\Omega := B_R \smallsetminus {B}_{R - 1}$. The well-known Bogovskiĭ Lemma~\ref{the Bogovskiĭ lemma} applied to $f=-\operatorname{div}(\phi \mathbf{v})$ yields that
\begin{subequations} \label{2024.12.30.2}
\begin{align}
&\|\mathbf{v}_{c}(\cdot,t)\|_{{W}^{k,q}( B_R \smallsetminus {B}_{R - 1})} \le C \| \mathbf{v}(\cdot,t)\|_{{W}^{k,q}(B_R \smallsetminus {B}_{R - 1}) } && \text{ for }1 < q < \infty,\\
&\|\Delta \mathbf{v}_{c}(\cdot,t)\|_{L^q( B_R \smallsetminus {B}_{R - 1})} \le C \|\Delta \mathbf{v}(\cdot,t)\|_{{L}^{q}( B_R \smallsetminus {B}_{R - 1}) }&& \text{ for }1 < q < \infty,\\
 \text{ and }\quad &\|{\partial }_{t}\mathbf{v}_{c}(\cdot,t)\|_{L^q( B_R \smallsetminus {B}_{R - 1})} \le C {\|\partial_{t} \mathbf{v}(\cdot,t)\|}_{L^q( B_R \smallsetminus {B}_{R - 1}) } && \text{ for }1 < q < \infty.
\end{align}
\end{subequations}
This together with the definition~ \eqref{2024.11.11} of $\mathbf{r}$ and estimates~\eqref{2024.11.29.1}, \eqref{2024.12.31.7} for ${\bf v}$ implies that
\begin{equation}
\left\|D^k \mathbf{r} (\cdot ,t)\right\|_{L^q( B_R )} \le  \begin{cases} C R^{-(n+1+k-n/q)},\\
 C t^{-(n+1+k- \frac{n}{q}) /2}\label{2024.12.1.1}
\end{cases}
\end{equation} 
for $1 < q < \infty$.

Moreover, using the estimate~\eqref{2024.11.29.1} of $D^k \mathbf{v}$ in the annular region, we obtain that 
\begin{small}
\begin{align}
\|{F}_{2}\|_{L^q( B_R )}= \|{F}_{2}\|_{L^q( B_R \smallsetminus {B}_{R - 1})} &= \|{\partial }_{t} \mathbf{v}_{c}- \Delta \mathbf{v}_{c} + ( \mathbf{v}_{c} \cdot \nabla ) \mathbf{v}_{c} + ( \mathbf{v}_{c} \cdot \nabla ) ( \phi \mathbf{v}) + ( \phi \mathbf{v} \cdot \nabla ) \mathbf{v}_{c}\|_{L^q( B_R \smallsetminus {B}_{R - 1})}\notag\\
&\le  C \left\{\|\partial_{t}\mathbf{v}_{c}\|_{L^q( B_R \smallsetminus {B}_{R - 1})}+ R^{-[n+3-(n-1)/q]}\right\}.
%\le& C (\|\partial_{t}\mathbf{v}_{c}\|_{L^2( B_R \smallsetminus {B}_{R - 1})}+ \|\Delta \mathbf{v}_{c}\|_{L^2( B_R \smallsetminus {B}_{R - 1})} +  \|\mathbf{v}_{c}\|_{{W}^{1,4}( B_R \smallsetminus {B}_{R - 1})} \| \mathbf{v}\|_{{W}^{1,4}( B_R \smallsetminus {B}_{R - 1}) } + \|\mathbf{v}_{c}\|_{{W}^{1,4}( B_R \smallsetminus {B}_{R - 1})}^{2})\notag\\
%\le& C (\|\partial_{t}\mathbf{v}\|_{L^2( B_R \smallsetminus {B}_{R - 1})}+\| \mathbf{v}\|_{{H}^{2}( B_R \smallsetminus {B}_{R - 1})}+\| \mathbf{v}\|_{{W}^{1,4}( B_R \smallsetminus {B}_{R - 1})}^2).
\end{align}
\end{small}
It follows that
\begin{align*}
\|\partial_{t}\mathbf{v}\|_{L^q( B_R \smallsetminus {B}_{R - 1})} &\le \|\Delta \mathbf{v}\|_{L^q( B_R \smallsetminus {B}_{R - 1})} + \| ( \mathbf{v} \cdot \nabla ) \mathbf{v}\|_{L^q( B_R \smallsetminus {B}_{R - 1})} + \|\nabla p \|_{L^q( B_R \smallsetminus {B}_{R - 1})}\\
&\le C R^{-[n+1-(n-1)/q]},
\end{align*}
where Equations~\eqref{2024.12.08.4} and \eqref{2024.11.29.1} are used to control the terms  $\na p$ and $\p_t {\bf v}$, respectively. Thus,
\begin{align} \label{2024.12.15.1}
\|{F}_{2}(\cdot ,t)\|_{L^q( B_R \smallsetminus {B}_{R - 1})} \le C R^{-[n+1-(n-1)/q]},
\end{align}
which together with the estimate~\eqref{2024.11.18.6} for $F_1$ gives us
\begin{align} \label{2024.11.29.4}
\|({F}_{1}+{F}_{2})(\cdot ,t)\|_{L^q( B_R)} = \|({F}_{1}+{F}_{2})(\cdot ,t)\|_{L^q( B_R \smallsetminus {B}_{R - 1})} \le C R^{-[n+1-(n-1)/q]}.
\end{align}

Similar arguments lead to the spacetime decay of higher-order derivatives of $F_2$ and $F_1+F_2$:
\begin{equation}\label{2024.12.23.3}
\left\|D^k {F}_{2}(\cdot ,t)\right\|_{L^q( B_R )} + \left\|D^k({F}_{1}+{F}_{2})(\cdot ,t)\right\|_{L^q( B_R)} \le C R^{-[n+1+k-(n-1)/q]}
\end{equation}
and
\begin{equation}\label{2025.1.7.7}
\left\|D^k ({F}_{1}+{F}_{2})(\cdot ,t)\right\|_{L^q( B_R)} \leq C t^{-(n+1+k- \frac{n}{q}) /2}.
\end{equation}
In all the inequalities in this subsection, $q \in (1,\infty)$ is arbitrary.

\section{Iterative estimates for NSE on $B_R$ with small forcing}\label{sec: iteration}

This section is devoted to the construction of the solutions of NSE on $B_R$ with a small forcing term. Recall from \eqref{2024.11.19.1} that $\mathbf{w} := \mathbf{r} + \tilde{\mathbf{v}}$ and $\pi := \phi p + \bar{p}$. Our goal here is to find $\left(\tilde{\mathbf{v}}, \bar{p}\right)$ via Eq.~\eqref{2024.11.2.1}. This shall lead to $(\mathbf{w},\pi)$ that satisfies the estimates in Theorem~\ref{main thm}.

To begin with, as $\phi \mathbf{v} \equiv \mathbf{v}$ on ${B}_{R - 1}$, and $\mathbf{v}_{c}\equiv 0$ on $B_{R-1}$, we see that $\mathbf{r} = \phi \mathbf{v} + \mathbf{v}_{c} = \mathbf{v}$ on $B_{R-1}$, which solves the NSE~\eqref{2024.10.31.1} thereon. Thus $\tilde{\mathbf{v}}$ is supported on $B_R \setminus {B}_{R - 1}$.

We shall establish the existence of solution for Eq.~\eqref{2024.11.2.1}, via the fixed-point theorem. Inspired by the formal solution in Eq.~\eqref{2024.11.18.2}, we seek a solution of the form
\begin{align}
\tilde{\mathbf{v}}(t)=e^{-t A} \mathbf{v}_{0}-\int_0^t e^{-(t-s) A} \mathbb{P} \bigg\{( \tilde{\mathbf{v}} \cdot \nabla ) \tilde{\mathbf{v}}+F+ p \nabla \phi+ ( \tilde{\mathbf{v}} \cdot \nabla ) \mathbf{r} + ( \mathbf{r} \cdot \nabla ) \tilde{\mathbf{v}}\bigg\}(s) \,\mathrm{d} s.
\end{align}
For this purpose, we design the iteration scheme as follows: with a solenoidal vector field ${\bf u}_i$ given for $i \in \mathbb{N}$ and ${\bf u}_0=0$, we define ${\bf u}_{i+1}$ via
\begin{align}\label{iteration}
\mathbf{u}_{i+1} (t) & =e^{-t A} \mathbf{v}_{0}-\int_0^t e^{-(t-s) A} \mathbb{P} \bigg\{( \mathbf{u}_i \cdot \nabla ) \mathbf{u}_i +F + p \nabla \phi+ ( \mathbf{u}_i  \cdot \nabla ) \mathbf{r} + ( \mathbf{r} \cdot \nabla ) \mathbf{u}_i \bigg\}(s) \,\mathrm{d} s\notag\\
&= e^{-t A} \mathbf{v}_{0}-\int_0^t e^{-(t-s) A} \mathbb{P} (F-\phi\nabla p )(s) \,\mathrm{d} s\nonumber\\
&\qquad -\int_0^t e^{-(t-s) A} \mathbb{P}\bigg\{\nabla\cdot(p \phi \,\id+ \mathbf{u}_i \otimes \mathbf{u}_i + \mathbf{u}_i  \otimes \mathbf{r} + \mathbf{r} \otimes \mathbf{u}_i )\bigg\}(s) \,\mathrm{d} s\notag\\
&= e^{-t A} \mathbf{v}_{0}+\int_0^t\!\int_{\mathbb{R}^n} K_1(x-y,t-s) (\phi \nabla p)(y,s) \,\mathrm{d} y \,\mathrm{d} s \nonumber\\
&\qquad -\int_0^t\!\int_{\mathbb{R}^n} K_1(x-y,t-s) F(y,s) \,\mathrm{d} y \,\mathrm{d} s\notag\\
&\qquad-\int_0^t\!\int_{\mathbb{R}^n} K_2(x-y,t-s) ( p \phi\, \id+ \mathbf{u}_i \otimes \mathbf{u}_i + \mathbf{u}_i  \otimes \mathbf{r} + \mathbf{r} \otimes \mathbf{u}_i )(y,s) \,\mathrm{d} y \,\mathrm{d} s.
\end{align}
In the second equality we use the identities $\operatorname{div}\mathbf{u}_i =\operatorname{div} \mathbf{r}=0$. Here $K_1$ and $K_2$ are the kernel functions  associated with $e^{-tA}\mathbb{P}=(e^{-tA}\mathbb{P}_{jk})$ and $\nabla e^{-t A} \mathbb{P}=(\partial_{x_\ell} e^{-t A} \mathbb{P}_{j k})$, respectively; and $\id$ is the identity matrix.

The crucial tool for controlling  $\mathbf{u}_{i+1}$ are the following potential estimates for kernels.

\begin{lemma} \label{lemma 4.1}
The kernel function $K_1(x,t)$ of $e^{-tA}\mathbb{P}$ satisfies
\begin{align} 
&\left |D^k K_1(x,t)\right | \leq C |x|^{-\alpha }t^{-\beta /2}\qquad \qquad  \text{for any } \alpha, \beta \geq 0 \text{ with }\alpha + \beta = n+k  \label{2024.11.18.8}
%&|\partial_{i}^{k} K(x,t)| \leq C |x|^{-\alpha }t^{-\beta /2},\quad &&\forall \alpha \text{ and }\beta \geq 0 \text{ with }\alpha + \beta = n+k, \label{2024.11.18.9} 
\end{align}
and
%\begin{subequations} 
\begin{align} \label{2024.11.18.7}
\| K_1(\cdot ,t)\|_{L^q(\R^n \smallsetminus B_1)} \leq C' t^{-\left(n- \frac{n}{q}\right) /2} \qquad \qquad \text{ for any } 1 \leq q \leq \infty.
%&\| K(\cdot ,t)\|_{n - \frac{n}{q},1,\infty } \leq C \|K(\cdot ,t)\|_{q} \leq C t^{-(n- \frac{n}{q}) /2}&&\text{for} \quad {\frac{n}{n + 1} < q < 1} ,\\
%&\|K(\cdot ,t)\|_{-1,1,\infty } \leq C \|K(\cdot ,t)\|_{\frac{n}{n + 1},w} \leq C.
\end{align}
%\end{subequations}
Here $k\in\{0,1,2,\ldots\}$, $C=C(n,k)$, and $C'=C'(n,k,q)$.
\end{lemma}
\begin{proof}
See the Appendix.
\end{proof}

\begin{lemma} \label{lemma 4.2}
The kernel function $K_2(x,t)$ of $\nabla e^{-t A} \mathbb{P}$ satisfies
\begin{align} 
&\left |D^k K_2(x,t)\right | \leq C |x|^{-\alpha }t^{-\beta /2}\qquad \qquad  \text{for any } \alpha, \beta \geq 0 \text{ with }\alpha + \beta = n+k+1   \label{2024.11.18.11}
%&|\partial_{i}^{k} K(x,t)| \leq C |x|^{-\alpha }t^{-\beta /2},\quad &&\forall \alpha \text{ and }\beta \geq 0 \text{ with }\alpha + \beta = n+k, \label{2024.11.18.9} 
\end{align}
and
%\begin{subequations} 
\begin{align} \label{2024.11.18.10}
\| K_2(\cdot ,t)\|_{L^q(\R^n \smallsetminus B_1)} \leq C' t^{-\left(n+1- \frac{n}{q}\right) /2}\qquad\qquad
\text{ for any } 1 \leq q \leq \infty.
%&\| K(\cdot ,t)\|_{n - \frac{n}{q},1,\infty } \leq C \|K(\cdot ,t)\|_{q} \leq C t^{-(n+1- \frac{n}{q}) /2}&&\text{for} \quad {\frac{n}{n + 1} < q < 1} ,\\
%&\|K(\cdot ,t)\|_{-1,1,\infty } \leq C \|K(\cdot ,t)\|_{\frac{n}{n + 1},w} \leq C.
\end{align}
%\end{subequations}
Here $k\in\{0,1,2,\ldots\}$, $C=C(n,k)$, and $C'=C'(n,k,q)$.
\end{lemma}
\begin{proof}
See \cite [Lemma 2.1] {MR1815476} and \cite [Corollary 4.1] {MR3288012}.
\end{proof}

Now we are ready to establish the decay properties of the sequence $\mathbf{u}_i$, which is essential to prove the convergence of the above iteration scheme. Note that ${\bf u}_{i+1}$ remains solenoidal by construction in Eq.~\eqref{iteration}. %Also, in view of the second paragraph in this section, we shall only consider $\mathbf{u}_i$ supported in $B_R \smallsetminus {B}_{R - 1}$.

\begin{lemma} \label{lemma 4.3}
Let $q \in (1,\infty)$ and $k\in\{0,1,2,\ldots\}$. If $\mathbf{v}_{0} \in {C}^{k}(\R^n; \R^n)$ is solenoidal and satisfies
\begin{equation}\label{2024.11.29.2}
\left \|e^{-tA} D^k \mathbf{v}_{0}(\cdot ,t)\right \|_{L^q( B_R)} \le 
\begin{cases}
C R^{-\left(n-1+k-\frac{n}{q}\right)},  \\
C t^{-\left(n-1-\varepsilon+k- \frac{n}{q}\right) /2}. 
    \end{cases}
\end{equation}
and if ${\bf u}_i$ is assumed to satisfy
\begin{equation}\label{2024.12.1.2}
\left \|D^k \mathbf{u}_i(\cdot ,t)\right \|_{L^q( B_R)} \leq      \begin{cases} 
C R^{-\left(n-1+k-\frac{n}{q}\right)}, \\
C t^{-\left(n-1-\varepsilon+k- \frac{n}{q}\right) /2},
    \end{cases}
\end{equation}
then it holds for $\mathbf{u}_{i+1}$ that
\begin{equation}\label{2024.11.29.3}
\left \|D^k \mathbf{u}_{i+1}(\cdot ,t)\right \|_{L^q(B_R)} \leq\begin{cases}
C R^{-\left(n-1+k-\frac{n}{q}\right)},  \\
C t^{-\left(n-1-\varepsilon+k- \frac{n}{q}\right) /2}. 
    \end{cases}
\end{equation}
\end{lemma}

Before presenting the proof, we remark on the validity of the assumption in Eq.~\eqref{2024.11.29.2}. Indeed, it is proved in  \cite[Theorem 4.3]{MR3288012} that 
$$
\left | e^{-tA} D^k \mathbf{v}_{0}( x) \right | \leq C ( 1 + | x|)^{-\alpha } ( 1 + t)^{-\beta /2}\qquad \text{ for any } \alpha,\beta \geq 0 \text{ with }\alpha + \beta = n+1+k.
$$
Thus, by adapting arguments used for deriving Eqs.~\eqref{2024.11.29.1} and \eqref{2024.12.31.7} we obtain \eqref{2024.11.29.2}.

\begin{proof}[Proof of Lemma~\ref{lemma 4.3}]
Consider as in \cite[Theorem 3.1]{MR1815476} the following:
\begin{small}
\begin{align}
\mathbf{u}_{i+1} (x,t) &= e^{-t A} \mathbf{v}_{0}+\int_0^t\!\int_{\mathbb{R}^n} K_1(x-y,t-s) (\phi\nabla p )(y,s) \,\mathrm{d} y\, \mathrm{d} s-\int_0^t\!\int_{\mathbb{R}^n} K_1(x-y,t-s) F(y,s) \,\mathrm{d} y \,\mathrm{d} s\notag\\
&\quad-\int_0^t\!\int_{\mathbb{R}^n} K_2(x-y,t-s) \Big( p \phi\, \id+ \mathbf{u}_i \otimes \mathbf{u}_i + \mathbf{u}_i  \otimes \mathbf{r} + \mathbf{r} \otimes \mathbf{u}_i \Big)(y,s) \,\mathrm{d} y \, \mathrm{d} s\notag\\
&= e^{-t A} \mathbf{v}_{0}+\int_0^t\!\int_{\mathbb{R}^n} K_1(x-y,t-s) (\phi\nabla p)(y,s) \,\mathrm{d} y \,\mathrm{d} s-\int_0^t\!\int_{\mathbb{R}^n} K_1(x-y,t-s) F(y,s) \,\mathrm{d} y\, \mathrm{d} s\notag\\
&\quad-\int_0^t \!\left\{ \int_{|x - y| \leq |x| /2} + \int_{|x - y| > |x|/2}\right\}K_2(x-y,t-s)  (p \phi\, \id) (y,s) \,\mathrm{d} y\, \mathrm{d} s\notag\\
&\quad-\int_0^t \!\int_{\mathbb{R}^n}  K_2(x-y,t-s) \Big( \mathbf{u}_i \otimes \mathbf{u}_i + \mathbf{u}_i  \otimes \mathbf{r} + \mathbf{r} \otimes \mathbf{u}_i\Big)(y,s) \,\mathrm{d} y \,\mathrm{d} s\notag\\
&=: I_1+I_2+I_3+I_4+I_5.\label{2025.1.17.2}
\end{align}
\end{small}

%As remarked at the beginning of this section, $\tilde{\bf v}$, which is the formal limit $\lim_{i \to \infty} {\bf u}_i$, is supported on the annulus $B_R \smallsetminus B_{R-1}$. 

\smallskip
\noindent
{\bf Estimate for $I_1$.}  It follows directly from the decay assumption~\eqref{2024.11.29.2}.

\smallskip
\noindent
{\bf Estimate for $I_2$.}  Using the estimate of $K_1$ and $\nabla p$ in \eqref{2024.11.18.8} and \eqref{2024.12.08.3}, respectively, and adopting an approach similar to that for \eqref{2024.11.29.5}, one has for any $\e \in (0,1)$ that
\begin{align*}
|I_2(x,t)|&\leq \int_0^t\!\int_{\mathbb{R}^n} \left| K_1(x-y,t-s) (\phi\nabla p )(y,s)\right| \,\mathrm{d} y \,\mathrm{d} s\notag\\
&\leq C \int_0^t \!\left\{\int_{|x - y| \leq |x| /2}+\int_{|y| < |x|/2} \right\}( t - s)^{-\varepsilon} | x - y|^{2\varepsilon - n} ( 1 + |y| )^{-(n-1+2\varepsilon)} ( 1 + s)^{-(1-\varepsilon)}\,\mathrm{d} y\, \mathrm{d} s\notag\\
&\quad +C \int_0^t \!\int_{|y| > \frac{3|x|}{2}}( t - s)^{-(1+\varepsilon)/2} | x - y|^{-(n-1-\varepsilon)} ( 1 + |y| )^{-(n+\varepsilon)} ( 1 + s)^{-(1-\varepsilon)/2}\,\mathrm{d} y \,\mathrm{d} s.
\end{align*}
Let us further \emph{claim} that
\begin{equation}\label{2025.1.7.4}
    |I_2(x,t)| \leq C(1+|x|)^{-\alpha(n-1)} (1+t)^{-\frac{\beta(n-1-\varepsilon)}{2}}
\end{equation}
holds true for any $\alpha, \beta \geq 0$ with $\alpha + \beta = 1$ and for arbitrarily small $\varepsilon>0$.

To see the \emph{claim}~\eqref{2025.1.7.4}, we shall check for $\alpha = 0$ and $\beta = 0$ separately; the remaining cases then hold by interpolation. 

\smallskip
\noindent
\underline{The case $\alpha=0$.} We tackle the time integral as follows: for $\lambda, \mu \geq 0$ with $\lambda+\mu = 1$, one has
\begin{align*}
\int_0^t ( t - s)^{-\lambda} ( 1 + s)^{-\mu}\, \mathrm{d} s &=\left(\int_0^{t/2}+\int_{t/2}^t\right) ( t - s)^{-\lambda} ( 1 + s)^{-\mu} \,\mathrm{d} s\\
&\le \int_0^{t/2} ( t/2)^{-\lambda}  ( 1 + s)^{-\mu} \mathrm{d} s+ \int_{t/2}^t (  t - s)^{-\lambda}  s^{-\mu} \,\mathrm{d} s \\
&= \int_0^{t/2} ( t/2)^{-\lambda}  ( 1 + s)^{-\mu} \,\mathrm{d} s+ \int_0^{t/2} \tilde{s}^{-\lambda} (  t - \tilde{s})^{-\mu}\, \mathrm{d} \tilde{s} \\
&\le \int_0^{t/2} ( t/2)^{-\lambda}  ( 1 + s)^{-\mu} \,\mathrm{d} s+ \int_0^{t/2} \tilde{s}^{-\lambda} ( t/2)^{-\mu}\, \mathrm{d} \tilde{s}\\
&\le 2.
\end{align*}

For the spatial integrals, we estimate over regions $\left\{y:\,|x - y| \leq |x| /2\right\}$ and $\left\{y:\,|x - y| > |x| /2\right\}$, respectively. In the former region one has  $ |x| \leq 2 |y|$, thus 
\begin{align*}
&\int_0^t \!\int_{|x - y| \leq |x| /2} ( t - s)^{-\varepsilon} | x - y|^{2\varepsilon - n} ( 1 + |y| )^{-(n-1+2\varepsilon)} ( 1 + s)^{-(1-\varepsilon)}\,\mathrm{d} y\, \mathrm{d} s\\
&\qquad \leq C |x|^{2\varepsilon} ( 1 + |x|)^{-(n-1+2\varepsilon)} \\
&\qquad \leq C ( 1 +|x|)^{-(n-1)},
\end{align*}
with $\varepsilon$ is a small positive constant. In the latter region, observe that $|x - y|^{-j} \le (|x|/2)^{-j}$, and hence for $|x|\geq 1$ that
\begin{align*} 
&\int_0^t \!\int_{|y| < |x|/2} ( t - s)^{-\varepsilon} | x - y|^{2\varepsilon - n} ( 1 + |y| )^{-(n-1+2\varepsilon)} ( 1 + s)^{-(1-\varepsilon)}\,\mathrm{d} y \,\mathrm{d} s \\
&\qquad\leq C |x|^{-(n-2\varepsilon)} ( 1 + |x|)^{1-2\varepsilon}\\
&\qquad\leq C ( 1 +|x|)^{-(n-1)}
\end{align*}
and
\begin{align*} 
&\int_0^t \!\int_{|y| > \frac{3|x|}{2}}( t - s)^{-(1+\varepsilon)/2} | x - y|^{-(n-1-\varepsilon)} ( 1 + |y| )^{-(n+\varepsilon)} ( 1 + s)^{-(1-\varepsilon)/2}\,\mathrm{d} y \,\mathrm{d} s\\
&\qquad\leq C |x|^{-(n-1-\varepsilon)} ( 1 + |x|)^{-\varepsilon} \\
&\qquad \leq C ( 1 +|x|)^{-(n-1)}.
\end{align*}
The proof for the case $\alpha=0$ is now complete. \\

\smallskip
\noindent
\underline{The case $\beta = 0$.} We divide the time interval into $[0,t/2]$ and $[t/2,t]$. Without loss of generality, assume $t\geq 1$. By choosing $\varepsilon$ suitably small, one has that
\begin{align*} 
&\int_0^{t/2}\!\int_{\mathbb{R}^n} |K_1(x-y,t-s) (\phi\nabla p)(y,s)| \,\mathrm{d} y\, \mathrm{d} s\\
&\qquad =\int_0^{t/2}\!\int_{B_R} |K_1(x-y,t-s) (\phi\nabla p )|(y,s) \,\mathrm{d} y \,\mathrm{d} s\\
&\qquad\leq\int_0^{t/2} \!\int_{B_R} ( t - s)^{-(n - 1-\varepsilon) /2}( t - s)^{-(1+\varepsilon) /2} ( 1 + |y| )^{-(n+\varepsilon)} ( 1 + s)^{-(1-\varepsilon)/2}\,\mathrm{d} y \,\mathrm{d} s \\
&\qquad\leq\left(\frac{t}{2}\right)^{-(n - 1-\varepsilon) /2} \int_0^{t/2} \!\int_{B_R} ( t - s)^{-(1+\varepsilon)/2} ( 1 + |y| )^{-(n+\varepsilon)} ( 1 + s)^{-(1-\varepsilon)/2}\,\mathrm{d} y \,\mathrm{d} s \\
&\qquad\leq C ( 1+t)^{-(n - 1-\varepsilon) /2}.
\end{align*}

On the other hand, by kernel estimate in \eqref{2024.11.18.7} with $q = 1$, pointwise estimate of $\nabla p$ in \eqref{2024.12.08.3}, and Young’s inequality, we have
\begin{align*} 
\int_{t/2}^t\!\int_{\mathbb{R}^n} |K_1(x-y,t-s) ( \phi\nabla p)(y,s) |\,\mathrm{d} y\, \mathrm{d} s &\le \int_{t/2}^t ( 1 + s)^{-(1 + n)/2 }\mathrm{d} s \\
&\leq C ( 1 + t)^{-(n - 1) /2}.
\end{align*}
This completes the proof of the case $\beta=0$. Hence, the pointwise estimate~\eqref{2025.1.7.4} follows from an interpolation argument.

The inequality~\eqref{2025.1.7.4} together with the calculations in \eqref{2024.11.29.5} and \eqref{2024.12.30.6} yield that, for $ 1 \le q <\infty$ and $\e>0$ arbitrarily small,
\begin{equation}\label{2024.12.01.3}
\|I_2(\cdot ,t)\|_{L^q(B_R)}
\le
\begin{cases} C R^{-\left(n-1-\frac{n}{q}\right)},\\
C t^{-\left(n-1-\varepsilon- \frac{n}{q}\right) /2}.
\end{cases}
\end{equation}

\smallskip
\noindent
{\bf Estimate for $I_3$.} To this end, we recall Eq.~\eqref{2024.11.29.4} for the estimate for $F$ (which is supported on $B_R \smallsetminus B_{R-1}$) and apply H\"{o}lder's inequality to deduce that 
\begin{align}
|I_3|&= \left|\int_0^t\!\int_{\mathbb{R}^n} K_1(x-y,t-s) F(y,s) \,\mathrm{d} y \,\mathrm{d} s\right|\nonumber\\
&=\left|\int_0^t\!\int_{B_R \smallsetminus {B}_{R - 1}} K_1(x-y,t-s) F(y,s) \,\mathrm{d} y \,\mathrm{d} s\right|\notag\\
&\le C \int_0^t \|K_1(x-\cdot,t-s)\|_{L^{\tilde{q}'}(B_R \smallsetminus {B}_{R - 1})}\|F(\cdot,s)\|_{L^{\tilde{q}}( B_R \smallsetminus {B}_{R - 1})}\,\mathrm{d} s\notag\\
&\le C \int_0^t R^{-[n+1-(n-1)/\tilde{q}]} \left[\left( \int_{|x - y| \leq |x| /2} + \int_{|x - y| > |x|/2}\right) | x - y|^{-\tilde{q}'(n-2-\varepsilon)} \,\mathrm{d} y \right]^{1/\tilde{q}'} (t-s)^{-(1+\varepsilon/2)}\,\mathrm{d} s \notag.
\end{align}
Here, $1/\tilde{q}+1/\tilde{q}'=1$ for $\tilde{q} \in (1,\infty)$, and the arbitrarily small parameter $\e>0$ is introduced to ensure the convergence of the integral $\int_0^t (t-s)^{-(1+\varepsilon/2)}\,\mathrm{d} s$. A direct adaptation of the estimates for $I_2$ leads to
\begin{align*}
|I_3(x,t)|\leq C |x|^{-(n-2-\varepsilon)+\frac{n}{\tilde{q}'}} R^{-\left(n+1-\frac{n-1}{\tilde{q}}\right)}.
\end{align*}
Thus %Since $I_3$ is supported in $B_R \smallsetminus B_{R - 1}$, we replicate the calculation in \eqref{2024.11.29.5} to derive the desired result:
\begin{align}
\|I_3(\cdot ,t)\|_{L^q(B_R)} \le C R^{-\left(n-1-\frac{n}{q}\right)} \quad \text{ for } 1 < q < \infty.
\end{align}

The temporal decay of $I_3$ is estimated as follows:
\begin{align*}
|I_3|&= \left|\int_0^t\!\int_{\mathbb{R}^n} K_1(x-y,t-s) F(y,s) \,\mathrm{d} y\, \mathrm{d} s\right|\\
&=\left|\int_0^{t/2}\!\int_{\mathbb{R}^n} K_1(x-y,t-s) F(y,s) \,\mathrm{d} y \,\mathrm{d} s\right|+\left|\int_{t/2}^t\!\int_{\mathbb{R}^n} K_1(x-y,t-s) F(y,s) \,\mathrm{d} y \,\mathrm{d} s\right|\\
&=:I_{3,1}+I_{3,2}.
\end{align*}
Thanks to the pointwise estimate~\eqref{2024.11.18.8} for $K_1$ and the spacetime estimates~\eqref{2025.1.7.7}-\eqref{2025.1.7.7} for $F$,
\begin{align*}
I_{3,1}&= \left|\int_0^{t/2}\!\int_{\mathbb{R}^n} K_1(x-y,t-s) F(y,s) \,\mathrm{d} y \,\mathrm{d} s\right|\notag\\
&\le C \int_0^{t/2} \|K_1(x-\cdot)\|_{L^{q}( \mathbb{R}^n)}\|F(\cdot ,s)\|_{L^{q'}( B_R \smallsetminus {B}_{R - 1})}\,\mathrm{d} s\notag\\
&\le C \left\{\int_0^{t/2}(t-s)^{-(n-1-\varepsilon- \frac{n}{q}) /2}(t-s)^{-1/2}s^{-1/2} \,\mathrm{d} s\right\}\\
&\quad\quad \times R^{-(n-n/q')} \left[\left( \int_{|x - y| \leq |x| /2} + \int_{|x - y| > |x|/2}\right) | x - y|^{-(q\varepsilon+n)} \,\mathrm{d} y \right]^{1/q}  \\
&\le C  t^{-\left(n-1-\varepsilon- \frac{n}{q}\right) /2} R^{-\left(n-\frac{n}{q'}\right)} |x|^{-\varepsilon}
\end{align*}
and
\begin{align*}
I_{3,2}&= \left|\int_{t/2}^t\!\int_{\mathbb{R}^n} K_1(x-y,t-s) F(y,s) \,\mathrm{d} y \,\mathrm{d} s\right|\notag\\
&\le C \int_{t/2}^t \|K_1(x-\cdot)\|_{L^{q'}( \mathbb{R}^n)}\|F(\cdot ,s)\|_{L^{q}( B_R \smallsetminus {B}_{R - 1})}\,\mathrm{d} s\notag\\
&\le C \int_{t/2}^t  s^{-\left(n+1- \frac{n}{q}\right) /2} \left[\left( \int_{|x - y| \leq |x| /2} + \int_{|x - y| > |x|/2}\right) | x - y|^{-q'n} \,\mathrm{d} y \right]^{1/q'} \,\mathrm{d} s \notag\\
&\le C t^{-(n-1- \frac{n}{q}) /2} |x|^{-n/q}.
\end{align*}
%Here $\tilde{q}=\e^{-1}$ as above. 
We thus conclude that
\begin{align}\label{I3, May25}
\|I_3(\cdot ,t)\|_{L^q( B_R )} \leq C t^{-\left(n-1-\varepsilon- \frac{n}{q}\right) /2} \quad \text{ for } 1 < q < \infty \text{ and arbitrarily small } \e>0.
\end{align}

\smallskip
\noindent
{\bf Estimate for $I_4$.}  By using the pointwise decay properties of $p$ in \eqref{2024.12.08.2} and the estimates for $\nabla e^{-t A} \mathbb{P}$ in \eqref{2024.11.18.11}, and employing an identical approach to that used for $I_2$, we infer that
\begin{align*}
|I_4| &=\left|\int_0^t \!\left( \int_{|x - y| \leq |x| /2} + \int_{|x - y| > |x|/2}\right)K_2(x-y,t-s) ( p \phi\, \id)(y,s) \,\mathrm{d} y \,\mathrm{d} s\right|\\
&\leq C(1+|x|)^{-\alpha(n-1)} (1+t)^{-\frac{\beta(n-1-\varepsilon)}{2}},
\end{align*}
for any $\alpha, \beta \geq 0$ with $\alpha + \beta = 1$ and arbitrarily small $\varepsilon>0$.
 Therefore, for $ 1 \le q < \infty$,
\begin{equation}\label{2025.1.17.3}
\|I_4(\cdot ,t)\|_{L^q(B_R)} \le \begin{cases}
C R^{-\left(n-1-\frac{n}{q}\right)},\\
C t^{-\left(n-1-\varepsilon- \frac{n}{q}\right) /2}.
\end{cases}
\end{equation} 

\smallskip
\noindent
{\bf Estimate for $I_5$.}  As for $I_5 =\left|\int_0^t \!\int_{\mathbb{R}^n}  K_2(x-y,t-s) ( \mathbf{u}_i \otimes \mathbf{u}_i + \mathbf{u}_i  \otimes \mathbf{r} + \mathbf{r} \otimes \mathbf{u}_i )(y,s) \,\mathrm{d} y \,\mathrm{d} s\right|$, note that by the estimates for $\mathbf{r}$ in \eqref{2024.12.1.1} and $\mathbf{u}_i$ in \eqref{2024.12.1.2}, we have that
\begin{equation*}
\|(\mathbf{u}_i \otimes \mathbf{u}_i + \mathbf{u}_i  \otimes \mathbf{r} + \mathbf{r} \otimes \mathbf{u}_i )(\cdot ,t)\|_{L^{q}( B_R)} \le
\begin{cases}
C R^{-(2n-2-n/q)},\\
C t^{-(2n-2-\varepsilon- \frac{n}{q}) /2}.
\end{cases}
\end{equation*}
Analogous arguments to those for $I_3$ yield that
\begin{equation}\label{2024.12.01.4}
\|I_5(\cdot ,t)\|_{L^q( B_R)}  \le \begin{cases} 
C R^{-(2n-3-n/q)}, \\
C t^{-(2n-3-\varepsilon- \frac{n}{q}) /2}
\end{cases}
\end{equation} 
for any $1 < q < \infty$.

\smallskip 

Combining the estimates in Eqs.~\eqref{2024.12.01.3}-\eqref{2024.12.01.4} and the decay assumption for the homogeneous Stokes system in Eq.~\eqref{2024.11.29.2}, we may now complete the proof for the case $k=0$.

For the case $k \ge 1$, we repeat the estimates for the $k=0$ case above and apply the kernel estimates in Lemmas~\ref{lemma 4.1} and \ref{lemma 4.2} for higher order derivatives. More precisely, we just replace the estimates for $p$, $\mathbf{r}$, and $F$ themselves in the above arguments with the derivative estimates for $p$ in \eqref{2024.12.08.3}-\eqref{2025.1.7.8}, for $\mathbf{r}$ in \eqref{2024.12.1.1}, and for $F$ in \eqref{2024.12.23.3}-\eqref{2025.1.7.7}. Then, together with the decay assumptions specified in Eqs.~\eqref{2024.11.29.2}–\eqref{2024.12.1.2}, we complete the proof of Lemma~\ref{lemma 4.3}.   \end{proof}

%To complete the proof of existence, we invoke the following fixed-point theorem for bilinear forms from Cannone \cite{MR1688096}:
%\begin{theorem}
%Let $\mathcal{X}$ be a Banach space equipped with norm $\| \cdot \|_{\mathcal{X}}$ , and let $B : \mathcal{X} \times \mathcal{X} \rightarrow \mathcal{X}$ denoted a bilinear map  satisfying the inequality
%$$
%\| B(\mathbf{u},\mathbf{v}) \|_{\mathcal{X}} \leq \eta \| \mathbf{u}\|_{\mathcal{X}}\| \mathbf{v}\|_{\mathcal{X}},
%$$
%for all $(\mathbf{u},\mathbf{v}) \in \mathcal{X} \times \mathcal{X}$.
%If $\mathbf{u}_{0} \in \mathcal{X}$ satisfies the condition $4\eta \|\mathbf{u}_{0}\|_{\mathcal{X}} < 1$, then the equation
%$$
%\mathbf{u} = \mathbf{u}_{0} + B(\mathbf{u},\mathbf{u})
%$$
%admits a solution $\mathbf{u} \in \mathcal{X}$. Moreover, this solution is unique one that satisfies $\| \mathbf{u}\|_{\mathcal{X}} \leq 2 \| \mathbf{u}_{0}\|_{\mathcal{X}}$.
%\end{theorem}

\begin{remark}
In the above estimates~\eqref{I3, May25}--\eqref{2024.12.01.4} for $I_3$--$I_5$, the constants are chosen independently of $\e$. More precisely, for a sufficiently small parameter $\e_0>0$, the above estimates hold for constants depending only on $\e_0$, $k$, $q$, $R_0$,  and $n$ whenever $\e \in ]0,\e_0[$. We shall fix such an $\e_0>0$ once and for all.  
\end{remark}

\section{Proof of Main Theorem~\ref{main thm}}\label{sec: proof}
With the above preparations at hand, we are ready to conclude our main result. Throughout the proof below, we write $\left({\bf u}_i, {\bf u}_{i+1}\right) = \left(\underline{\mathbf{u}},\mathbf{u}\right)$ for simplicity.

\begin{proof}[Proof of Theorem~\ref{main thm}] 
Recall that the desired solution $(\mathbf{w}, \pi)$ of Eq.~\eqref{2024.11.2.1} on $B_R$ is given by
$$
\mathbf{w} := \phi \mathbf{v} + \mathbf{v}_{c} + \tilde{\mathbf{v}} \quad \text{ and } \quad \pi := \phi p + \bar{p},
$$
where $(\mathbf{v},p)$ is the NSE solution on $\R^n$, $\phi$ is a given cutoff function, and ${\bf v}_{c}$ is the correction term accounting for the non-divergence-free property of $\phi {\bf v}$. It remains to prove the existence of $(\mathbf{w}, \pi)$ solving for Eq.~\eqref{2024.11.2.1}, together with desired decay estimates for $({\bf v}-{\bf w}, p-\pi)$. 

Throughout this proof, fix an arbitrary $q \in (1,\infty)$ once and for all. The uniqueness of strong solution is classical in the literature, so we focus only on the existence issue.

We apply a standard fixed point argument. The function space in consideration is the following spacetime-weighted Sobolev space $\mathcal{X}$ motivated by Schonbek \cite{MR3288012}: define 
\begin{equation}\label{def, X}
\mathcal{X}\equiv \mathcal{X}(k,q,R,n):= \left\{{\bf u} \in W^{k,q}_0(B_R; \R^n):\, {\rm div}({\bf u})=0 \text{ and } \|{\bf u}\|_{\mathcal{X}} < \infty \right\},
\end{equation}
where 
\begin{align*}
\|\mathbf{u}\|_{\mathcal{X}}= \mathop{\sup }\limits_{\substack{{\alpha ,\beta \geq 0,\;\alpha + \beta = 1;}\\{m = 0,1,2,3,\ldots k;\; t \in [0,T]} }} R^{\alpha \left(n-1+m-\frac{n}{q}\right)}\, t^{\beta \left(n-1-\varepsilon+m- \frac{n}{q}\right) /2} \left\|D^m \mathbf{u}(\cdot,t) \right\|_{L^q( B_R)}. 
\end{align*}

Consider the mapping $\T : \mathcal{X}  \rightarrow \mathcal{X}$:
\begin{align}\label{T def}
\T \underline{\mathbf{u}} \equiv \mathbf{u} &:=    e^{-t A} \mathbf{v}_{0}-\int_0^t e^{-(t-s) A} \mathbb{P} (F-\phi\nabla p)(s) \,\mathrm{d} s\nonumber\\
&\qquad -\int_0^t \nabla e^{-(t-s) A} \mathbb{P}\, \Big\{p \phi \,\id+ \underline{\mathbf{u}} \otimes \underline{\mathbf{u}} + \underline{\mathbf{u}}  \otimes \mathbf{r} + \mathbf{r} \otimes \underline{\mathbf{u}}\Big\}(s) \,\mathrm{d} s.
\end{align} 
We \emph{claim} the following bounds: 
\begin{align*}
&\left|\int_0^t \nabla e^{-(t-s) A} \mathbb{P}\, ( \mathbf{r} \otimes \underline{\mathbf{u}})(s) \,\mathrm{d} s\right| \leq C_1 R^{-1}\|\mathbf{r}\|_{\mathcal{X}}\|\underline{\mathbf{u}}\|_{\mathcal{X}},\\
&\left|\int_0^t \nabla e^{-(t-s) A} \mathbb{P}\, (\underline{\mathbf{u}}\otimes \mathbf{r} )(s) \,\mathrm{d} s\right| \leq C_1 R^{-1}\|\underline{\mathbf{u}}\|_{\mathcal{X}}\|\mathbf{r}\|_{\mathcal{X}},\\
&\left|\int_0^t \nabla e^{-(t-s) A} \mathbb{P}\, (\underline{\mathbf{u}}\otimes \underline{\mathbf{u}})(s) \,\mathrm{d} s\right| \leq C_1 R^{-1}\|\underline{\mathbf{u}}\|_{\mathcal{X}}^2,
\end{align*}
from which it follows that
\begin{align}
\|\T \underline{\mathbf{u}}\|_{\mathcal{X}} \le& \left\|e^{-t A} \mathbf{v}_{0}\right\|_{\mathcal{X}}+ \left\|\int_0^t e^{-(t-s) A} \mathbb{P} (F-\phi\nabla p )(s) \,\mathrm{d} s\right\|_{\mathcal{X}}+\left\|\int_0^t \nabla e^{-(t-s) A} \mathbb{P}\, (p \phi \id) (s) \,\mathrm{d} s\right\|_{\mathcal{X}}\notag\\
&+  C_1 R^{-1}\|\mathbf{r}\|_{\mathcal{X}}\|\underline{\mathbf{u}}\|_{\mathcal{X}}+C_1R^{-1}\|\underline{\mathbf{u}}\|_{\mathcal{X}}^2.\label{2025.1.20}
\end{align}
The constant $C_1$, as in Remark~\ref{remark on const}, is independent of $R$ and $t$.

To see the \emph{claim}, let us only prove the first inequality, as the remaining two are analogous. 
Fix $m \in \{0,1,\ldots,k\}$ for $qm>n$ as in the definition~\eqref{def, X} for the space $\mathcal{X}\equiv \mathcal{X}(k,q,R,n)$, and fix any $q,q' \in (1,\infty)$ with $\frac{1}{q}+\frac{1}{q'}=1$. Noticing that $\mathbf{r} \otimes \underline{\mathbf{u}}$ is supported in $B_R$, and using the kernel estimate~\eqref{2024.11.18.11} for $K_2$ in Lemma~\ref{lemma 4.2}, we argue as for $I_3$ above to obtain that
\begin{align*}
&\left|\int_0^t \nabla e^{-(t-s) A} \mathbb{P}\,  D^m ( \mathbf{r} \otimes \underline{\mathbf{u}})(s) \,\mathrm{d} s\right|\notag\\
&\qquad=\left|\int_0^t  \!\int_{\mathbb{R}^n}  K_2(x-y,t-s)  D^m ( \mathbf{r} \otimes \underline{\mathbf{u}})(y,s) \,\mathrm{d} y \,\mathrm{d} s\right|\\
&\qquad\leq C \int_0^t \|K_2(x-\cdot,t-s)\|_{L^{q'}( B_R )}\|D^m ( \mathbf{r} \otimes \underline{\mathbf{u}})(\cdot,s)\|_{L^{q}( B_R )}\mathrm{d} s\notag\\
&\qquad\leq C \int_0^t \|D^m ( \mathbf{r} \otimes \underline{\mathbf{u}})(\cdot,s)\|_{L^{q}( B_R )}   \nonumber\\
&\qquad\qquad\qquad\times\left[\left( \int_{|x - y| \leq |x| /2} + \int_{|x - y| > |x|/2}\right) | x - y|^{-q'(n-1-\varepsilon)} \,\mathrm{d} y \right]^{1/q'} (t-s)^{-(1+\varepsilon/2)}\mathrm{d} s \notag\\
&\qquad\leq C \|D^m ( \mathbf{r} \otimes \underline{\mathbf{u}})(\cdot,s)\|_{L^{q}( B_R )} |x|^{-(n-1-\varepsilon)+n/q'}\\
&\qquad\leq C \Big\{ \|\mathbf{r}(\cdot,s)\|_{L^\infty} \|\underline{\mathbf{u}}(\cdot,s)\|_{W^{m,q}}+ \|\underline{\mathbf{u}}(\cdot,s)\|_{L^\infty} \|\mathbf{r}(\cdot,s)\|_{W^{m,q}}\Big\} |x|^{-(n-1-\varepsilon)+n/q'}\\
&\qquad\leq C \|\mathbf{r} (\cdot,s)\|_{W^{m,q}( B_R )} \|\underline{\mathbf{u}}(\cdot,s)\|_{W^{m,q}( B_R )} |x|^{-(n-1-\varepsilon)+n/q'}.
\end{align*} 
The penultimate inequality follows from the Sobolev--Gagliardo--Nirenberg--Moser inequality, and the final line holds by the Sobolev embedding $W^{m,q}\emb L^\infty$ over domains in $\R^n$ for $m>\frac{n}{q}$.

As a consequence, by setting $\varepsilon=1/q$ one deduces
\begin{align*}
&\left\|\int_0^t \nabla e^{-(t-s) A} \mathbb{P}\,  D^m ( \mathbf{r} \otimes \underline{\mathbf{u}})(s) \,\mathrm{d} s\right\|_{\mathcal{X}}\\
&\qquad \leq C R^{-\left(n-1+m-\frac{n}{q}\right)} \|\mathbf{r}\|_{\mathcal{X}}\|\underline{\mathbf{u}}\|_{\mathcal{X}} \cdot R\\
&\qquad \leq C_1 R^{-1}\|\mathbf{r}\|_{\mathcal{X}}\|\underline{\mathbf{u}}\|_{\mathcal{X}},
\end{align*}
thereby the \emph{claim} and hence  \eqref{2025.1.20} follows.

We now continue from Eq.~\eqref{2025.1.20}. The first three terms on the right-hand side of it correspond to $I_2, I_3$, and $I_4$ (recall Eq.~\eqref{2025.1.17.2}), and are estimated in Eqs.~\eqref{2024.12.01.3}-\eqref{2025.1.17.3}. For the fourth term, we bound it by 
\begin{align*}
\|\mathbf{r}\|_{\mathcal{X}} \leq C R^{-2},
\end{align*}
thanks to the estimate for $\mathbf{r}$ in Eq.~\eqref{2024.12.1.1}. Hence,
\begin{align*}
\|\T \underline{\mathbf{u}}\|_{\mathcal{X}} \le \left\|e^{-t A} \mathbf{v}_{0}\right\|_{\mathcal{X}}+  C_3+  C_2 R^{-3}\|\underline{\mathbf{u}}\|_{\mathcal{X}}+C_1 R^{-1}\|\underline{\mathbf{u}}\|_{\mathcal{X}}^2,
\end{align*}
with $C_i$ depending only on $q,k,$ and $n$.

By selecting sufficiently small initial data such that
$$
\left\|e^{-t A} 
\mathbf{v}_{0}\right\|_{\mathcal{X}} < \frac{(1-C_2R^{-3})^2}{4 C_1R^{-1}}-C_3,
$$
we have that $\|\T \underline{\mathbf{u}}\|_{\mathcal{X}} \le L$ whenever \begin{align} \label{2025.1.17.1}
\|\underline{\mathbf{u}}\|_{\mathcal{X}}\le L := \frac{1-C_2R^{-3}-\sqrt{(1-C_2R^{-3})^2-4C_1R^{-1}(\left\|e^{-t A} \mathbf{v}_{0}\right\|_{\mathcal{X}} +C_3)}}{2 C_1R^{-1}}.
\end{align}

We may now conclude that $\T$ is a contraction. The nonlinear terms are treated similarly as for \eqref{2025.1.20}. Indeed, we have
\begin{align*}
&\left\|\T\underline{\mathbf{u}}_1-\T\underline{\mathbf{u}}_2\right\|_{\mathcal{X}}\nonumber\\
&\qquad =\left\|\int_0^t \nabla e^{-(t-s) A} \mathbb{P}\, \Big( (\underline{\mathbf{u}}_1+\underline{\mathbf{u}}_2) \otimes (\underline{\mathbf{u}}_1-\underline{\mathbf{u}}_2) + (\underline{\mathbf{u}}_1-\underline{\mathbf{u}}_2)  \otimes \mathbf{r} + \mathbf{r} \otimes (\underline{\mathbf{u}}_1-\underline{\mathbf{u}}_2) \Big)(s) \,\mathrm{d} s \right\|_{\mathcal{X}} \nonumber\\
&\qquad \leq C_1 R^{-1} (\|\underline{\mathbf{u}}_1\|_{\mathcal{X}}+\|\underline{\mathbf{u}}_2\|_{\mathcal{X}}+\|\mathbf{r}\|_{\mathcal{X}})\|\underline{\mathbf{u}}_1-\underline{\mathbf{u}}_2\|_{\mathcal{X}}\nonumber\\
&\qquad \leq  \left(2 C_1R^{-1} L+C_2R^{-3}\right) \|\underline{\mathbf{u}}_1-\underline{\mathbf{u}}_2\|_{\mathcal{X}}\nonumber\\
&\qquad < \|\underline{\mathbf{u}}_1-\underline{\mathbf{u}}_2\|_{\mathcal{X}},
\end{align*}
where the last line follows from \eqref{2025.1.17.1}. Therefore, applying the Banach fixed point theorem, we deduce the existence of a solution $\tilde{\mathbf{v}}$ to Eq.~\eqref{2024.11.2.1} with 
\begin{equation}\label{2024.12.30.3}
\left\|D^k \tilde{\mathbf{v}}(\cdot,t) \right\|_{L^q( B_R )} \leq \begin{cases}  
C R^{-\left(n-1+k-\frac{n}{q}\right)},\\
C t^{-\left(n-1-\varepsilon+k- \frac{n}{q}\right) /2} \qquad\text{for any $1 < q < \infty$.}
\end{cases}
\end{equation} 

For $\bar{p}$, we employ the same approach as in \S\ref{section 1.1}: taking divergence to both sides of Eq.~\eqref{2025.01.03.1} and exploiting the fundamental solution for Laplacian, one obtains that
\begin{align*}
&\bar{p}(x,t)=C(n) \int_{\mathbb{R}^{n}} \frac{1}{|x-y|^{n-2}}\left[\operatorname{div}(F-\phi\nabla p )+\operatorname{div}\operatorname{div}(p \phi \id+ \tilde{\mathbf{v}} \otimes \tilde{\mathbf{v}} + \tilde{\mathbf{v}}  \otimes \mathbf{r} + \mathbf{r} \otimes \tilde{\mathbf{v}}) \right](y,t) \,\mathrm{d}y.
\end{align*}
Employing the calculations for $I_3$ outlined in Lemma \ref{lemma 4.3}, and combining it with the estimates for $p$ in \eqref{2024.12.08.4}-\eqref{2025.1.7.8}, $\mathbf{r}$ in \eqref{2024.12.1.1}, $F$ in \eqref{2024.12.23.3}-\eqref{2025.1.7.7}, and $\tilde{\mathbf{v}}$ in \eqref{2024.12.30.3}, we then conclude that
\begin{equation*}
\|\bar{p } (\cdot ,t)\|_{L^q( B_R )}=\|\bar{p } (\cdot ,t)\|_{L^q( B_R \smallsetminus {B}_{R - 1})} \leq 
\begin{cases} 
C R^{-\left(n-\frac{n-1}{q}\right)},\\
C t^{-\left(n- \frac{n}{q}\right) /2}.
\end{cases}
\end{equation*} 
Similarly, for any $1 < q < \infty$,
\begin{equation} \label{2024.12.30.4}
\left\|D^k \bar{p } (\cdot ,t)\right\|_{L^q( B_R )} \leq \begin{cases}
C R^{-\left(n+k-\frac{n-1}{q}\right)} \\
C t^{-\left(n+k- \frac{n}{q}\right)/2}.
\end{cases}
\end{equation}

It is clear by construction that $(\mathbf{v},p)$ and $(\mathbf{w},\pi)$ coincide within $B_{R-1}$.
Applying the triangle inequality along with the estimates for $\phi \mathbf{v}$ in \eqref{2025.1.13.1}, $\mathbf{v}_{c}$ in \eqref{2024.12.30.2}, $\tilde{\mathbf{v}}$ in \eqref{2024.12.30.3}, $p$ in \eqref{2024.12.08.4}-\eqref{2025.1.7.8}, and $ \bar{p}$ in \eqref{2024.12.30.4}, we arrive at
\begin{equation*}
\left\| D^k ( \mathbf{v}- \mathbf{w} ) (\cdot ,t)\right\|_{L^q ( B_R) }\leq \begin{cases}
 C R^{-\left(n-1+k-\frac{n}{q}\right)},\\
 C t^{-\left(n-1-\varepsilon+k- \frac{n}{q}\right)/2},
\end{cases}
\end{equation*}
and 
\begin{equation*}
\left\| D^k (p-\pi) (\cdot ,t)\right\|_{L^q ( B_R) }\leq \begin{cases}
C R^{-\left(n+k-\frac{n-1}{q}\right)},\\
C t^{-\left(n+k- \frac{n}{q}\right) /2}.
\end{cases}
\end{equation*}
From here, Eq.~\eqref{2024.10.31.2} and hence Theorem~\ref{main thm} follows via an interpolation argument.    \end{proof}

\appendix
\section{Technical Lemmas}
In the Appendix, we present the Bogovskiĭ Lemma used for the correction term $\mathbf{v}_{c}$ and provide a detailed proof for the estimates of kernel functions. 

The following result, known as the Bogovskiĭ Lemma, is a classical result in the literature (Bogovskiĭ \cite{MR0553920, MR0631691}; see also Galdi \cite[Lemma III.3.1]{MR2808162}).
\begin{lemma} \label{the Bogovskiĭ lemma}
Let $\Omega \subset {\mathbb{R}}^{n}(n \ge 2)$ be a star-shaped domain with respect to ${B}_{r}$ --- that is, for all $x \in \Omega$ and $y \in {B}_{r}$, the line segment from $x$ to $y$ lies in $\Omega$. Then for any $q \in (1,\infty)$ and $f \in {W}^{k - 1,q} ( \Omega )$ with $\int_\Omega f = 0$, there exists $\mathbf{v} \in {W}_0^{k,q} ( {\Omega ;{\mathbb{R}}^{n}})$ such that $\operatorname{div} \mathbf{v} = f$. Moreover, 
\begin{align}
\| \mathbf{v}\|_{{W}^{k,q} ( \Omega ) } \le C \| f\|_{{W}^{k - 1,q}( \Omega ) },
\end{align}
where the constant $C$ depending on $n,k,q$ and the domain $\Omega$.
\end{lemma} 
In fact, letting $g \in {C}_{0}^{\infty } ( {B}_{1} )$ be a function such that $\int g = 1$, the vector field
\begin{align} \label{2024.11.7}
\mathbf{v} (x) := \int_{\Omega } f(y) \left( {\frac{x - y}{{| x - y|}^{n}} \int _{| x - y| }^{\infty } g\left( y + z\frac{x - y}{| x - y| }\right) {z}^{n-1} \mathrm{\;d}z}\right) \mathrm{d}y 
\end{align}
satisfies the assertion of the lemma. The result remains valid even if the domain $\Omega$ is not star-shaped, thanks to a partition of unity argument as in \cite{MR4316125}. 

We now turn to the potential estimates for integral kernels. Throughout the proof below, the constant $C$ depends on the dimension $n$ and the domain, while $c$ denotes a universal constant. We shall assume Lemma~\ref{lemma 4.2} (see \cite [Lemma 2.1] {MR1815476} and \cite [Corollary 4.1] {MR3288012} for a proof) and deduce Lemma~\ref{lemma 4.1}.

\begin{proof}[Proof of Lemma~\ref{lemma 4.1}]
The case $k=0$ in \eqref{2024.11.18.8} follows from an adaptation of \cite [Lemma 2.1] {MR1815476}. For the sake of completeness and clarity, we present the details here. Also, we shall prove for the case $k \ge 1$ by way of using  the inequality $
\left|D^k K_1(x,t)\right| \leq C \left|D^{k-1} K_2(x,t)\right|$, where $K_1(x,t)$ and $K_2(x,t)$ are the kernel functions of the linear operators $e^{-tA}\mathbb{P}=(e^{-tA}\mathbb{P}_{jk})$ and $\nabla e^{-t A} \mathbb{P}=(\partial_{\ell} e^{-t A} \mathbb{P}_{j k})$, respectively. 

Recall the Fourier transform:
$$
\hat{\mathbf{v}}(\xi)=\int e^{-i x \cdot \xi} \mathbf{v}(x) \,\mathrm{d} x \quad(i=\sqrt{-1}).
$$
The Leray projection $\mathbb{P}$ is expressed componentwise as
$$
(\widehat{\mathbb{P} \mathbf{v}})_j(\xi)=\sum_{k=1}^n \hat{\mathbb{P}}_{j k}(\xi) \hat{\mathbf{v}}_k(\xi)=\sum_{k=1}^n\left(\delta_{j k}+\frac{i \xi_j i \xi_k}{|\xi|^2}\right) \hat{\mathbf{v}}_k(\xi) .
$$
The Fourier transform of the kernel function $K(x, t)=\left(K_{j k}(x, t)\right)$ of the operator $e^{-t A} \mathbb{P}=(e^{-t A} \mathbb{P}_{j k})$ can thus be expressed as
\begin{align} \label{2024.12.23.1}
\hat{K}_{j k}(\xi, t)=e^{-t|\xi|^2}\left(\delta_{j k}+\frac{i \xi_j i \xi_k}{|\xi|^2}\right) \equiv \hat{K}_{j k}^1+\hat{K}_{j k}^2 .
\end{align}

Applying the inverse Fourier transform, we obtain
\begin{align}
K_{j k}^1(x, t)=&(2\pi)^{-n}\int e^{i \xi \cdot x} e^{-t|\xi|^2} \,\mathrm{d} \xi\notag\\
\stackrel{\tilde{\xi}=\sqrt{t} \xi}{=}&(2\pi)^{-n} \int e^{i \tilde{\xi}\cdot \frac{x}{\sqrt{t}}} e^{-|\tilde{\xi}|^2} t^{-n/2}\,\mathrm{d} \tilde{\xi}\notag\\
=&C t^{-n/2} e^{-\frac{c|x|^2}{t}}. \label{2024.12.10.3}
\end{align}
Hence, 
\begin{align*}
\left|K_{j k}^1(x, t)\right| \leq C t^{-n/ 2} e^{-c|x|^2 / t} \leq C t^{-n/ 2} \text{ and } \left|K_{j k}^1(x, t)\right| \leq C |x|^{-n}.
\end{align*}
An interpolation argument yields, for all non-negative $\alpha$ and $\beta$ with $\alpha + \beta = n$, that
\begin{align} \label{24.12.10.1}
\left|K_{j k}^1(x, t)\right| \leq C|x|^{-\alpha} t^{-\beta / 2}.
\end{align}

To estimate $K_{j k}^2(x, t)$, from the identity $
|\xi|^{-2}=\int_0^{\infty} e^{-s|\xi|^2} \,\dd s$ we deduce that
$$
\hat{K}_{j k}^2(\xi, t)=\int_0^{\infty} i \xi_j i \xi_k e^{-(s+t)|\xi|^2} \,\dd s=\int_t^{\infty} i \xi_j i \xi_k e^{-s|\xi|^2} \,\dd s.
$$
Taking the inverse Fourier transform, we obtain
$$
K_{j k}^2(x, t)=\int_t^{\infty} \partial_{x_j} \partial_{x_k} E_s(x) \,\dd s,
$$
where $E_s(x)=(4 \pi s)^{-n / 2} e^{-|x|^2 / 4 s}$.  Since
$$
\left|\partial_{x_j} \partial_{x_k} E_s(x)\right| \leq C s^{-(n+3) / 2} e^{-c|x|^2 / s},
$$
we infer that
\begin{align*}
\left |K_{j k}^2(x, t)\right| &\leq C \int_t^{\infty} s^{-(n+3) / 2} e^{-c|x|^2 / s} \,\dd s \\
&\leq C \int_t^{\infty} s^{-(n+3) / 2} \,\dd s\\
&=C t^{-(n+1) / 2}.
\end{align*}
In addition, a change of the variable $\tau:=|x|^2 / s$ shows that
\begin{align*}
\left |K_{j k}^2(x, t)\right | &\leq C \int_t^{\infty} s^{-(n+3) / 2} e^{-c|x|^2 / s} \,\dd s\\
&=C |x|^{-(n+1)} \int_0^{|x|^2 / t} \tau^{(n-1) / 2} e^{-c \tau} \,\dd \tau \\
&\leq C|x|^{-(n+1)}.
\end{align*}
As a consequence,
\begin{align} \label{24.12.10.2}
\left |K_{j k}^2(x, t)\right| \leq C|x|^{-\alpha} t^{-\beta / 2} \qquad \text{ for all }\alpha, \beta \geq 0 \text{ with }\alpha + \beta = n+1.
\end{align}
Eq.~\eqref{2024.11.18.8} now follows from Eqs.~\eqref{24.12.10.1} and \eqref{24.12.10.2}.

For the proof of the $L^q$-estimates \eqref{2024.11.18.7}, thanks to $K_{j k}^1(x, t)=C t^{-n/2} e^{-\frac{c|x|^2}{t}}$, we have that
\begin{align*}
\left \| K_{j k}^1(\cdot ,t)\right \|_{L^q}^q  &= C t^{-qn/2} \int_{|x|\ge 1} e^{-\frac{cq|x|^2}{t}} \mathrm d x \\
&\le Ct^{-qn/2} \int_1^\infty e^{-\frac{cq r^2}{t}} r^{n-1} \mathrm d r\\
&\stackrel{\zeta=r\sqrt{c/t} }{\le} Ct^{-qn/2} \int_1^\infty e^{-q \zeta^2} \left(\sqrt{\frac{t}{c}}\right)^{n-1} \zeta^{n-1} \left(\sqrt{\frac{t}{c}}\right) \mathrm d \zeta\\
&\le Ct^{-(qn-n)/2} \int_1^\infty e^{-q \zeta^2} \zeta^{n-1} \mathrm d \zeta\\
&\le Ct^{-(qn-n)/2}.
\end{align*}
On the other hand, $\left |K_{j k}^2(x, t)\right | \leq C|x|^{-\alpha} t^{-\beta / 2} $ in \eqref{24.12.10.2} implies
\begin{align*}
\left \| K_{j k}^2(\cdot ,t)\right \|_{L^q}^q  &\le Ct^{-(qn-n)/2} \int_{|x|\ge 1} |x|^{-(n+q)} \mathrm d x \le Ct^{-(qn-n)/2}. 
\end{align*}
These finish the proof.
\end{proof}

\medskip

\noindent
{\bf Acknowledgement}. The research of SL is supported by NSFC Projects 12201399, 12331008, and 12411530065, Young Elite Scientists Sponsorship Program by CAST 2023QNRC001, the National Key Research $\&$ Development Program 2023YFA1010900 and 2024YFA1014900,  Shanghai Rising-Star Program 24QA2703600,  and the Shanghai Frontier Research Institute for Modern Analysis. The research of XS is partially supported by the National Key Research $\&$ Development Programs 2023YFA1010900  and 2024YFA1014900.

Both authors are indebted to Prof.~Feng Xie for insightful discussions. We also thank one anonymous referee for constructive remarks on an earlier version of the manuscript.

\medskip
\noindent
{\bf Competing Interests Statement}. We declare that there is no conflict of interests involved.

\end{document}